\documentclass[a4paper,12pt]{amsart}
\usepackage{amsfonts}
\usepackage{mathrsfs}
\usepackage{amsmath,amssymb,latexsym,amsfonts,amscd, yfonts, amsthm, pifont}
\usepackage{wasysym}
\usepackage{cases}
\usepackage{bm}
\usepackage{tikz}
\usepackage{enumitem}
\usepackage[all]{xy}
\allowdisplaybreaks[4]
 \usepackage{graphicx}

\newtheorem{thm}{Theorem}[section]
\newtheorem{lem}[thm]{Lemma}
\newtheorem{cor}[thm]{Corollary}

\newtheorem{prop}[thm]{Proposition}

\newtheorem{conj}[thm]{Conjecture}

\theoremstyle{definition}

\newtheorem{definition}[thm]{Definition}

\newtheorem{example}[thm]{Example}

\theoremstyle{remark}
\newtheorem{remark}[thm]{Remark}

\numberwithin{equation}{section}

\newcommand{\bC}{{\mathbb C}}

\newcommand{\bQ}{{\mathbb Q}}

\newcommand{\rounddown}[1]{\lfloor{#1}\rfloor}

\newcommand\OO{{\mathcal{O}}}

\newcommand\ZZ{{\mathbb{Z}}}

\newcommand\vol{\text{\rm vol}}

\newcommand\td{{\rm{td}}}
\newcommand\tp{{\rm{tp}}}
\newcommand\sigmabar{{\overline{\sigma}}}

\newcommand\ch{{\rm{ch}}}
\newcommand\RR{{\rm{RR}}}
\newcommand\RW{{\rm{RW}}}
\newcommand\Sym{{\rm{Sym}}}
\newcommand\cB{\mathcal{B}}
\newcommand\hcB{\hat{\mathcal{B}}}
\newcommand\bw{{\bf w}}

\title{Positivity of Riemann--Roch polynomials and Todd classes of hyperk\"{a}hler manifolds}
\date{\today}

\author{Chen Jiang}
\address{Shanghai Center for Mathematical Sciences, Fudan University, Jiangwan Campus, 2005 Songhu Road, Shanghai, 200438, China}
\email{chenjiang@fudan.edu.cn}

\begin{document}
\begin{abstract} 
For a hyperk\"{a}hler manifold $X$ of dimension $2n$, Huybrechts showed that there are constants $a_0, a_2, \dots, a_{2n}$ such that
$$\chi(L) =\sum_{i=0}^n\frac{a_{2i}}{(2i)!}q_X(c_1(L))^{i}$$
for any line bundle $L$ on $X$, where $q_X$ is the Beauville--Bogomolov--Fujiki quadratic form of $X$. Here the polynomial $\sum_{i=0}^n\frac{a_{2i}}{(2i)!}q^{i}$ is called the Riemann--Roch polynomial of $X$.

In this paper, we show that all coefficients of the Riemann--Roch polynomial of $X$ are positive.
This confirms a conjecture proposed by Cao and the author, which implies Kawamata's effective non-vanishing conjecture for projective hyperk\"{a}hler manifolds. It also confirms a question of Riess on strict monotonicity of Riemann--Roch polynomials.

In order to estimate the coefficients of the Riemann--Roch polynomial, we produce a Lefschetz-type decomposition of $\td^{1/2}(X)$, the root of the Todd genus of $X$, via the Rozansky--Witten theory following the ideas of Hitchin, Sawon, and Nieper-Wi{\ss}kirchen.

\end{abstract}
\maketitle

\pagestyle{myheadings} \markboth{\hfill C. Jiang
\hfill}{\hfill Positivity of Riemann--Roch polynomials and Todd classes of hyperk\"{a}hler manifolds \hfill}
\tableofcontents

\section{Introduction}
Throughout this paper, we work over the complex number field $\bC$.

A compact K\"{a}hler manifold $X$ is called
a {\it hyperk\"{a}hler manifold} or an {\it irreducible holomorphic symplectic manifold} if $X$ is simply connected and $H^0(X, \Omega^2_X)$ is spanned by an everywhere non-degenerate $2$-form. Hyperk\"{a}hler manifolds are higher-dimensional analogues of K3 surfaces and appear to be a very important class of manifolds with $c_1=0$. 
Their rich geometry attracts much attention from different areas of mathematics.
The only known examples are (up to deformations): Hilbert schemes of points on K3 surfaces, generalized Kummer varieties (due to Beauville's construction \cite{beauville}), and $2$ examples in dimensions $6$ and $10$ constructed by O'Grady \cite{ogrady1, ogrady2}. 

The main goal of this paper is to study general properties of the Riemann--Roch polynomial and Todd classes of a hyperk\"{a}hler manifold.
\subsection{Positivity of Riemann--Roch polynomials}
For a hyperk\"{a}hler manifold $X$ of dimension $2n$, Huybrechts \cite{huybrechts} showed that there are constants $a_0, a_2, \dots, a_{2n}$ such that
$$\chi(L) =\sum_{i=0}^n\frac{a_{2i}}{(2i)!}q_X(c_1(L))^{i}$$
for any line bundle $L$ on $X$, where $q_X$ is the Beauville--Bogomolov--Fujiki quadratic form of $X$ (see Section~\ref{sec bbf form}). Here the polynomial $\RR_X(q)=\sum_{i=0}^n\frac{a_{2i}}{(2i)!}q^{i}$ is called the {\it Riemann--Roch polynomial} of $X$. Note that $\RR_X$ is a deformation invariant of $X$. To study the behavior of line bundles on hyperk\"{a}hler manifolds, it is crucial to have a good understanding of Riemann--Roch polynomials. 
In \cite{CJmathz}, Cao and the author conjectured that the coefficients of the Riemann--Roch polynomial are all non-negative for any projective hyperk\"{a}hler manifold, and proved it up to dimension $6$.
The main theorem of this paper is the following.

\begin{thm}\label{main thm1 RR>0}
Let $X$ be a hyperk\"{a}hler manifold. Then all coefficients of the Riemann--Roch polynomial $\RR_X(q)$ are positive.
\end{thm}

In fact, in Corollary~\ref{cor P>0}, we will have a more precise estimate on the lower bounds of the coefficients of $\RR_X$.
We remark that Nieper-Wi{\ss}kirchen \cite{nieper-jag} gave a closed formula for the coefficients $a_{2k}$ in terms of Chern numbers of $X$ by the Rozansky--Witten theory, but the expression is quite complicated and not sufficient to determine the positivity of coefficients.

\begin{example}\label{ex RR}The Riemann--Roch polynomials of known hyperk\"{a}hler manifolds are as the following:
\begin{enumerate}
\item If $X$ is a hyperk\"{a}hler manifold of dimension $2n$ deformation equivalent to the Hilbert scheme of $n$ points on a K3 surface or O'Grady's $10$-dimensional example, then $\RR_X(q)=\binom{q/2+n+1}{n}$ by \cite[Lemma 5.1]{egl} and \cite[Theorem 2]{ortiz}; 
\item If $X$ is a hyperk\"{a}hler manifold of dimension $2n$ deformation equivalent to a generalized Kummer variety or O'Grady's $6$-dimensional example, then $\RR_X(q)=(n+1)\binom{q/2+n}{n}$ by \cite[Lemma 5.2]{BN} and \cite[Theorem 2]{ortiz}. 
\end{enumerate}
\end{example}
From the known examples, we can observe that $\RR_X$ might satisfy more properties than positivity, so it is natural to raise up the following conjecture (the first one is a question asked by Ortiz).
\begin{conj}
Let $X$ be a hyperk\"{a}hler manifold. 
\begin{enumerate}
\item The sequence of coefficients of $\RR_X(q)$ is log concave.
\item All roots of $\RR_X(q)$ are negative real numbers.
\item More wildly, all roots of $\RR_X(q)$ are negative even integers forming an arithmetic sequence.
\end{enumerate}
\end{conj}

As applications of Theorem~\ref{main thm1 RR>0}, we give an affirmative answer to the conjecture of Cao and the author \cite{CJmathz} which leads to a solution of Kawamata's effective non-vanishing conjecture for projective hyperk\"{a}hler manifolds (Corollary~\ref{cj conj})
and also an affirmative answer to a question of Riess \cite{Riess} on the strict monotonicity of Riemann--Roch polynomials (Corollary~\ref{riess conj}).

\subsection{A Lefschetz-type decomposition of $\td^{1/2}(X)$ via the Rozansky--Witten theory}
To prove Theorem~\ref{main thm1 RR>0}, we need to have a good understanding of the Todd genus of a hyperk\"{a}hler manifold. In fact, as observed by Hitchin and Sawon \cite{hitchinsawon} and Nieper-Wi{\ss}kirchen \cite{nieper-jag}, the root of the Todd genus $\td^{1/2}(X)$ is a more interesting object, especially from the point of view of the Rozansky--Witten theory. Following their ideas, we use the Rozansky--Witten theory to produce a Lefschetz-type decomposition of $\td^{1/2}(X)$.

\begin{thm}[{=Proposition~\ref{prop primitive}+Theorem~\ref{td=sum tp}}]\label{main thm2 lef decomp}
Let $X$ be a hyperk\"{a}hler manifold of dimension $2n$ and fix a non-zero $\sigma\in H^0(X, \Omega^2_X)$. Consider $\lambda_\sigma=\frac{24n \int \exp(\sigma+\sigmabar)}{ \int c_{2}(X) \exp(\sigma+\sigmabar)}$.
For $0\leq k\leq n/2$, denote
$$
\tp_{2k}:=\sum_{i=0}^k\frac{(n-2k+1)!\td^{1/2}_{2i}\wedge(\sigma\sigmabar)^{k-i}}{(-\lambda_\sigma)^{k-i}(k-i)!(n-k-i+1)!}\in H^{4k}(X).$$
Then $\tp_{2k}$ is $(\sigma+\sigmabar)$-primitive for any $0\leq k\leq n/2$.
Furthermore, for any $0\leq k\leq n$, 
$$
\td^{1/2}_{2k}=\sum_{i=0}^{\min\{k, n-k\}} \frac{(n-k-i)!}{\lambda_\sigma^{k-i} (k-i)!(n-2i)!}\tp_{2i}\wedge(\sigma\sigmabar)^{k-i}.
$$
\end{thm}
Applying the Hodge--Riemann bilinear relation to the decomposition in Theorem~\ref{main thm2 lef decomp} 
will give a good estimate to $\int\td(X)\exp(\sigma+\sigmabar)$, which proves Theorem~\ref{main thm1 RR>0}.
Also this decomposition can recover known results due to Hitchin and Sawon \cite{hitchinsawon} and Nieper-Wi{\ss}kirchen \cite{nieper-jag} (see Corollary~\ref{cor cover nieper}).
Meanwhile, this result might be also interesting for its own sake to help us to study the cohomological structure of 
hyperk\"{a}hler manifolds.

We remark that the idea of using the Hodge--Riemann bilinear relation to prove Theorem~\ref{main thm1 RR>0} originates from \cite{CJmathz}, where we used a Lefschetz-type decomposition from \cite{guan} for $c_2(X)$ 
to show that Theorem~\ref{main thm1 RR>0} holds in dimension $6$. The decomposition there is given by the projection to the {\it Verbitsky component}, that is, the subalgebra $\text{\rm SH}^2(X)\subset H^*(X)$ generated by $H^2(X)$. However, this method only works for $c_2(X)$, and is not applicable to higher dimensions, as this decomposition is too coarse and we can not control the orthogonal complements in $\text{\rm SH}^2(X)^\perp.$

In order to prove the decomposition in Theorem~\ref{main thm2 lef decomp}, the key ingredient is to show the following result.

\begin{cor}[{=Corollary~\ref{lambda td}}]\label{main thm3 lambda td}Let $X$ be a hyperk\"{a}hler manifold and fix a non-zero $\sigma\in H^0(X, \Omega^2_X)$. Consider $\lambda_\sigma=\frac{24n \int \exp(\sigma+\sigmabar)}{ \int c_{2}(X) \exp(\sigma+\sigmabar)}$.
Then for any integer $k\geq 1$,
$$\Lambda_{\sigma/4}(\td^{1/2}_{2k})=\frac{1}{\lambda_\sigma}\td^{1/2}_{2k-2}\wedge\sigmabar.$$
\end{cor}
See Section~\ref{sec sl2} for the definition of $\Lambda_{\sigma/4}$. This result is proved using the Rozansky--Witten theory and the wheeling theorem, following the ideas of Hitchin and Sawon \cite{hitchinsawon} and Nieper-Wi{\ss}kirchen \cite{nieper-jag}. The key is to show a formula comparing the $\Lambda_{\sigma/4}$-action on Rozansky--Witten classes and the differential operator action on Jacobi diagrams (Theorem~\ref{RW diff}). Such a formula was originally observed by Nieper-Wi{\ss}kirchen in his thesis \cite{nieper-phd}. 

Finally, it is worth-mentioning that during the proof, we get the following by-product. It can be viewed as a counterpart of the result $\int\td^{1/2}(X)>0$ in \cite{hitchinsawon}, and it might have further applications to the topological structure of hyperk\"{a}hler manifolds.
\begin{cor}[{=Corollary~\ref{upper td1/2}}]
Let $X$ be a hyperk\"{a}hler manifold of dimension $2n>2$. 
Then
$\int\td^{1/2}(X)< 1.$
\end{cor}

This paper is organized as the following. In Section~\ref{sec pre}, we prepare necessary background knowledge.
In Section~\ref{sec RW}, we briefly recall the Rozansky--Witten theory and prove Theorem~\ref{RW diff} and Corollary~\ref{main thm3 lambda td}. In Section~\ref{sec lef dec}, we give a Lefschetz-type decomposition of $\td^{1/2}$ (Theorem~\ref{main thm2 lef decomp}). In Section~\ref{final sec}, we study the positivity of the Riemann--Roch polynomials, prove Theorem~\ref{main thm1 RR>0}, and give various applications.

\section{Preliminaries}\label{sec pre}

In this section, we collect basic knowledge on hyperk\"{a}hler manifolds. The readers may refer to \cite{huybrechts}.

\subsection{Complex structures}\label{sec IJK}

Let $X$ be a hyperk\"{a}hler manifold with a non-zero $\sigma\in H^0(X, \Omega^2_X)$ and a K\"{a}hler form $\omega$. Then there are $3$ complex structures $I,J,K$ on $X$ satisfying $K=IJ=-JI$ and a hyperk\"{a}hler metric $g$ compatible with all of them with corresponding K\"{a}hler forms $\omega=\omega_I, \omega_J, \omega_K.$
Up to a scalar, we may assume that $\sigma=\omega_J+\sqrt{-1}\omega_K$.

\subsection{Beauville--Bogomolov--Fujiki form and Riemann--Roch polynomial}\label{sec bbf form}
 Beauville \cite{beauville}, Bogomolov \cite{bogomolov}, and Fujiki \cite{fujiki} proved that there exists a quadratic form $q_X: H^2(X, \mathbb{R}) \to \mathbb{R}$ and a constant $c_X\in \mathbb{Q}_+$ such that for all $\alpha\in H^2(X, \mathbb{R})$,
$$\int\alpha^{2n}= c_X\cdot q_X(\alpha)^n. $$
The above equation determines $c_X$ and $q_X$ uniquely if assuming:
\begin{enumerate} 
\item $q_X$ is a primitive integral quadratic form on $H^2(X, \ZZ)$;
\item $q_X(\sigma + \overline{\sigma}) > 0$ for $0\neq \sigma \in H^{2,0}(X)$.
\end{enumerate}
Here $q_X$ and $c_X$ are called the {\it Beauville--Bogomolov--Fujiki form} and the {\it Fujiki constant} of $X$ respectively.

Recall the following important result by Fujiki \cite{fujiki} (see also \cite[Corollary 23.17]{gross} for a generalization).
\begin{thm}[{\cite{fujiki}, \cite[Corollary 23.17]{gross}}]\label{fujiki result}
Let $X$ be a hyperk\"{a}hler manifold of dimension $2n$. Assume that $\alpha\in H^{4j}(X,\mathbb{R})$ is of type $(2j, 2j)$ on all small deformations of $X$. Then there exists a constant ${\bf C}(\alpha)\in\mathbb{R}$ depending only on $\alpha$ such that
$$ \int \alpha\beta^{2n-2j}={\bf C}({\alpha})\cdot q_{X}(\beta)^{n-j}$$
for all $\beta\in H^{2}(X,\mathbb{R})$.
\end{thm}
 A direct application of this result (cf. \cite[1.11]{huybrechts}) is that, for a line bundle $L$ on $X$, the
Hirzebruch--Riemann--Roch formula gives
$$\chi(X,L)=\sum_{i=0}^{n}\frac{1}{(2i)!}\int\td_{2n-2i}(X)(c_{1}(L))^{2i}=\sum_{i=0}^{n}\frac{a_{2i}}{(2i)!}q_{X}\big(c_{1}(L)\big)^{i}, $$ 
where $$a_{2i}={\bf C}(\td_{2n-2i}(X)).$$
The polynomial $\RR_X(q):=\sum_{i=0}^n\frac{a_{2i}}{(2i)!}q^{i}$ is called the {\it Riemann--Roch polynomial} of $X$.
\subsection{Characteristic values}

For a hyperk\"{a}hler manifold, the characteristic value is defined by Nieper-Wi{\ss}kirchen, which is a quadratic form proportional to $q_X$. This quadratic form is more convenient than $q_X$ when playing with Rozansky--Witten classes and Riemann--Roch polynomials.

\begin{definition}[{\cite[Definition 17]{nieper-jag}}]\label{def lambda}
Let $X$ be a hyperk\"{a}hler manifold.
For any $\alpha\in H^2(X, \mathbb{R})$, Nieper-Wi{\ss}kirchen defined the {\it characteristic value} of $\alpha$,
$$
\lambda(\alpha):=\begin{cases}\frac{24n \int \exp(\alpha)}{ \int c_{2}(X) \exp(\alpha)} & \text{if well-defined;}\\ 0 & \text{otherwise.}\end{cases}
$$
For simplicity, we often denote $\lambda_\sigma:=\lambda(\sigma+\sigmabar)$.
\end{definition}

\begin{prop}[{cf. \cite[Proposition 10]{nieper-jag}}]\label{prop lambda=q}
$\lambda(\alpha)$ is a positive constant multiple of $q_X(\alpha)$, more precisely,
$$
\lambda(\alpha)=\frac{12c_X}{(2n-1){\bf C}(c_2(X))}q_X(\alpha).
$$
\end{prop}

Note that to study $\RR_X$, we may always view it as a polynomial in terms of $\lambda$. Here we remark that this multiple is positive (i.e., ${\bf C}(c_2(X))>0$) by Yau's solution to Calabi's conjecture \cite{yau} (cf. \cite[Proposition~3.11]{nieper-book}).

Recall that a line bundle $L$ on a projective manifold $X$ is said to be {\it nef} if $L\cdot C\geq 0$ for any curve $C\subset X$, moreover, it is said to be {\it big} if $L^{\dim X}>0$. We have the following easy lemma. 

\begin{lem}[{cf. \cite[1.10]{huybrechts}}]\label{lemma qL>0}Let $X$ be a hyperk\"{a}hler manifold and fix a non-zero $\sigma\in H^0(X, \Omega^2_X)$. Let $L$ be a line bundle on $X$.
\begin{enumerate}
\item If $X$ is projective and $L$ is nef and big, then $q_X(c_1(L))>0$ and $\lambda(L)>0$.
\item $\lambda_\sigma=\lambda(\sigma+\sigmabar)>0$.
\end{enumerate} 
\end{lem}
\subsection{Todd genus and Todd classes}\label{sec todd}
Let $X$ be a hyperk\"{a}hler manifold. It is well-known that all its odd Chern classes vanish. The {\it Todd genus} of $X$ (see \cite[(4.13)]{nieper-jag}) can be defined by
$$
\td(X)=\exp\left(-2\sum_{k=1}^\infty b_{2k}(2k)!\ch_{2k}(X)\right),
$$
where $\ch_{2k}(X)$ are the Chern characters of $X$ and $b_{2k}$ are the {\it modified Bernoulli numbers} defined by
$$ \sum_{k=0}^{\infty}b_{2k}x^{2k}=\frac{1}{2}\ln\frac{\sinh(x/2)}{x/2}.$$
The {\it square root of the Todd genus} of $X$ is defined by
$$
\td^{1/2}(X)=\exp\left(-\sum_{k=1}^\infty b_{2k}(2k)!\ch_{2k}(X)\right)
$$
which satisfies $(\td^{1/2}(X))^2=\td(X)$.
We will use $\td^{1/2}_{2k}=\td^{1/2}_{2k}(X)$ to denote the $2k$-th term of $\td^{1/2}(X)$.
For example, $\td^{1/2}_{0}=1$, $\td^{1/2}_{2}=\frac{1}{2}\td_2=\frac{1}{24}c_2(X)$, $\td^{1/2}_{4}=\frac{1}{5760}(7c_2^2(X)-4c_4(X))$.

Here we remark that, for hyperk\"{a}hler manifolds, rational Chern classes are determined by rational Pontrjagin classes (cf. \cite[Proposition 1.13]{nieper-book}), hence rational Chern classes (and hence Todd classes) are topological invariants of $X$. In particular, rational Chern classes are independent of complex structures.

\subsection{Some linear algebra on symplectic forms}
Let $\bf{k}$ be a field of characteristic zero, $V$ a $\bf{k}$-vector space, and $A$ a $\bf{k}$-algebra.

An element $\sigma \in {\bigwedge}^2V^*$ is called a {\it symplectic form} on $V$ if it defines a non-degenerate bilinear form on $V$.
Note that if a vector space $V$ admits a symplectic form $\sigma$, then its dimension is even, say $2n$. In this case we can always choose a symplectic basis $e_1,\dots, e_{2n}$ of $V$ such that $\sigma=\sum_{i=1}^n\vartheta^{2i-1}\wedge \vartheta^{2i}$, where $\vartheta^1,\dots, \vartheta^{2n}$ is the corresponding dual basis of $V^*$.

\begin{definition}\label{def contraction}
Let $V$ be a $\bf{k}$-vector space of dimension $2n$ admitting a symplectic form $\sigma$. The {\it contraction} by $\sigma$ is the map $\delta: \bigwedge V^*\otimes A \to \bigwedge V^*\otimes A$ define by
\begin{align*}
{}&\delta((\alpha_1\wedge\dots\wedge \alpha_l)\otimes a)\\
={}&\sum_{r=1}^n\left(\sum_{1\leq s<t\leq l}(-1)^{s+t-1}
\begin{aligned}
(\alpha_s(e_{2r-1})\alpha_t(e_{2r})-\alpha_s(e_{2r})\alpha_t(e_{2r-1}))\\
\cdot \, \alpha_1\wedge\dots\wedge\widehat{\alpha_s}\wedge\dots\wedge\widehat{\alpha_t}\wedge\dots\wedge \alpha_l
\end{aligned}\right)\otimes a
\end{align*}
for $\alpha_1,\dots, \alpha_l\in V^*$ and $a\in A$. Here $e_1,\dots, e_{2n}$ is a symplectic basis of $V$.
\end{definition}

Note that we can regard $\sigma$ as an operator $\sigma: \bigwedge V^*\otimes A \to \bigwedge V^*\otimes A$ by taking wedge product with $\sigma$, and consider the operator $\Pi: \bigwedge V^*\otimes A \to \bigwedge V^*\otimes A$
acting on ${\bigwedge}^p V^*\otimes A$ by multiplying with $p-n$.
 
\begin{prop}\label{prop local sl2}
Keep the above setting.
Then $(\sigma, \delta, \Pi)$ gives an $\mathfrak{sl}_2$-action on $\bigwedge V^*\otimes A$. Namely,
$$
[\sigma, \delta]=\Pi, \, [\Pi, \sigma]=2\sigma, \, [\Pi, \delta]=-2\delta. 
$$
\end{prop}

\begin{proof}
The proof is standard. To avoid tedious computations, we illustrate by the following example: it is easy to 
see that for any $a\in A$,
$$\delta(\sigma\otimes a)=\delta\bigg(\big(\sum_{i=1}^n\vartheta^{2i-1}\wedge \vartheta^{2i}\big)\otimes a\bigg)=n\otimes a,$$ hence $[\sigma, \delta](1\otimes a)=-\delta(\sigma\otimes a)=-n\otimes a =\Pi(1\otimes a).$
\end{proof}
\subsection{An $\mathfrak{sl}_2$-action on the cohomology of a hyperk\"{a}hler manifold}\label{sec sl2}
The cohomology of a hyperk\"{a}hler manifold has been studied by many authors, see for example \cite{fujiki, LL, Veb-coh}. In particular, there is a natural $\mathfrak{so}(4,1)$-action on the cohomology of a hyperk\"{a}hler manifold by \cite{Veb-so5}. In this paper, we mainly focus on a special $\mathfrak{sl}_2$-action induced by $\sigma$, which has been considered by Fujiki \cite{fujiki}, Huybrechts \cite{huy97}, and Nieper-Wi{\ss}kirchen \cite{nieper-phd}.

Let $X$ be a hyperk\"{a}hler manifold of dimension $2n$ and fix a non-zero $\sigma\in H^0(X, \Omega^2_X)$.
After a rescaling we may assume that $\sigma=\omega_J+\sqrt{-1}\omega_K$ as in Section~\ref{sec IJK}.
For any $0\leq p,q\leq 2n$, let $L_\sigma: H^{q}(X, \Omega_X^p)\to H^{q}(X, \Omega_X^{p+2})$ be the Lefschetz operator giving by the cup-product with $\sigma.$
Define $\Lambda_{\sigma/4}=*^{-1}\circ L_{\sigmabar/4}\circ *$ where $*$ is the Hodge operator associated to the K\"{a}hler metric $g$ compatible with the hyperk\"{a}hler structure of $X$, and define $\Pi: H^{q}(X, \Omega_X^p)\to H^{q}(X, \Omega_X^p)$ to be the map multiplying by $(p-n).$
Then we have
$$
[L_\sigma, \Lambda_{\sigma/4}]=\Pi, \, [\Pi, L_\sigma]=2L_\sigma, \, [\Pi, \Lambda_{\sigma/4}]=-2\Lambda_{\sigma/4}, 
$$
and hence $(L_\sigma, \Lambda_{\sigma/4}, \Pi)$ gives an $\mathfrak{sl}_2$-action on $H^{*}(X, \Omega_X^*)$ (cf. \cite[Proof of Theorem 6.3]{huy97}).

\begin{remark}\label{remark local sl2}
There is a natural local interpretation of this $\mathfrak{sl}_2$-action as the following: fix a point $x\in X$,
consider $V=\mathcal{T}_{X, x}$ the holomorphic tangent space and $A=\bigwedge \overline{\Omega}_{X, x}$, note that $\sigma_x$ is a symplectic form on $V$, then we can consider operators 
$(\sigma_x, \delta_x, \Pi_x)$ acting on $\bigwedge V^*\otimes A$ as in Definition~\ref{def contraction} and Proposition~\ref{prop local sl2}, which, on the level of cohomology, induce exactly $(L_\sigma, \Lambda_{\sigma/4}, \Pi)$ acting on $H^{*}(X, \Omega_X^*)$.
\end{remark}
\begin{definition}[{cf. \cite{fujiki}}]
A class $\alpha\in H^{*}(X, \Omega_X^k)$ is called {\it $\sigma$-primitive} if $\Lambda_{\sigma/4}(\alpha)=0$, which is equivalent to $L_{\sigma}^{n-k+1}(\alpha)=0$.
\end{definition}
The following lemma is standard.
\begin{lem}\label{sigma sigmabar commutes}
$[L_\sigma, L_\sigmabar]=[\Lambda_{\sigma/4}, L_\sigmabar]=0.$
\end{lem}
\begin{proof}
The first one is trivial. Let us consider the second one.
We may assume that $\sigma=\omega_J+\sqrt{-1}\omega_K$ as in Section~\ref{sec IJK}. Then by definition, 
$$
L_{\sigmabar}=L_{\omega_J}-\sqrt{-1}L_{\omega_K}, \Lambda_{\sigma/4}=\frac{1}{4}(\Lambda_{\omega_J}-\sqrt{-1}\Lambda_{\omega_K}).
$$
Then the conclusion follows immediately from \cite[(2.1)]{Veb-so5}.
\end{proof}
The following lemma is standard by the representation theory of $\mathfrak{sl}_2$, see for example \cite[Corollary 1.2.28]{huy-book}.
\begin{lem}\label{AL=L}For $\alpha\in H^{*}(X, \Omega_X^k)$ and $m\geq 1$,
$$[L_{\sigma}^m, \Lambda_{\sigma/4}](\alpha)=m(k-n+m-1)L_{\sigma}^{m-1}(\alpha).$$
In particular, if moreover $\alpha$ is $\sigma$-primitive, then $\Lambda_{\sigma/4}L_{\sigma}^{m}(\alpha)=m(n+1-k-m)L_{\sigma}^{m-1}(\alpha)$.
\end{lem}

\section{The Rozansky--Witten theory}\label{sec RW}
For a hyperk\"{a}hler manifold $X$ with a non-zero $\sigma\in H^0(X, \Omega^2_X)$, the Rozansky--Witten theory associates to every Jacobi diagram $\Gamma$ to a cohomology class $\RW_\sigma(\Gamma)$, which is due to Rozansky and Witten \cite{RW} and later developed by Kapranov \cite{Kapranov}. Later Hitchin and Sawon \cite{hitchinsawon} and Nieper-Wi{\ss}kirchen \cite{nieper-jag} discovered that this is a powerful tool to study the characteristic classes of hyperk\"{a}hler manifolds.

The main goal of this section is to apply the Rozansky--Witten theory to prove Corollary~\ref{lambda td}, which calculates $\Lambda_{\sigma/4}(\td^{1/2}_{2k})$ for a hyperk\"{a}hler manifold.
In order to explain the proof, we will briefly recall basic knowledge of the Rozansky--Witten theory, the readers may refer to \cite{Sawon-phd, nieper-phd, nieper-jag, nieper-book} for details. Most of the contents in this section are from \cite{nieper-jag}, while Proposition~\ref{RW diff local} and Theorem~\ref{RW diff} are originally claimed by Nieper-Wi{\ss}kirchen in \cite{nieper-phd}, and we provide a self-contained proof for the reader's convenience. 
\subsection{The graph homology space}
A
 {\it graph} is a collection of {\it vertices} connected by {\it edges},
 where every edge connects $2$ vertices. 
 A {\it flag} or a {\it half-edge} is an edge together with an adjacent vertex. So every edge consists of exactly $2$ flags, and every flag belongs to exactly $1$ vertex. Note that an edge or a vertex can be identified with the set of flags belonging to it. A vertex is called {\it univalent}, if there is only $1$ flag belonging to it, and it is called {\it trivalent}, if there are exactly $3$ flags belonging to it. 
A graph is called {\it vertex-oriented} if, for every vertex, a cyclic ordering of its flags is fixed.

\begin{definition}[Jacobi diagram]
A {\it Jacobi diagram} is a vertex-oriented graph with only univalent and trivalent vertices. 
A {\it trivalent Jacobi diagram} is a Jacobi diagram with no univalent vertices. The {\it degree} of a Jacobi diagram is the number of its vertices.
\end{definition}

When we draw a Jacobi diagram as a planar graph, we want the counter-clockwise ordering of the flags at each trivalent vertex in the drawing to be the same as the given cyclic ordering.

\begin{example}
\begin{enumerate}
\item The empty graph is a Jacobi diagram, which is denoted by $1$.
\item The Jacobi diagram consisting of $2$ univalent vertices connecting by $1$ edge is denoted by $\ell$, and called a {\it strut}.
\item The Jacobi diagram $\ominus$ consisting of $2$ trivalent vertices connecting by 3 edges is denoted by $\Theta$.

\item For each positive integer $k$, the {\it $2k$-wheel} $\bw_{2k}$ is a Jacobi diagram defined to be a closed path with $2k$ vertices and $2k$ edges, while every vertex has a third edge outside the closed path. So it contains $2k$ trivalent vertices and $2k$ univalent vertices. For example, the $8$-wheel $\bw_8$ looks like $\sun$.
\end{enumerate}
\end{example}

\begin{definition}[Graph homology space]
The space $\cB$ is defined to be the $\bQ$-vector space spanned by all Jacobi diagrams modulo the IHX relation and the anti-symmetry (AS) relation (see \cite{Thurston, nieper-jag} for definitions). The space $\cB'$ is defined to be the subspace of $\cB$ spanned by all Jacobi diagrams not containing $\ell$ as a component. The space $\cB^t$ is defined to be the subspace of $\cB$ spanned by all trivalent Jacobi diagrams. These spaces are graded by degrees, and bi-graded by the numbers of trivalent and univalent vertices. We denote $\cB_{k,l}$ to be the homogenous part of $\cB$ generated by Jacobi diagrams with $k$ trivalent and $l$ univalent vertices.
The completion of $\cB$ (resp. $\cB'$, $\cB^t$) with respect to the grading is denoted by $\hcB$ (resp. $\hcB'$, $\hcB^t$).
\end{definition}

There are 2 natural operations on the graph homology spaces.
\begin{definition}[Disjoint union]
The disjoint union of Jacobi diagrams induces a bilinear map
$$
\hcB\times \hcB \to \hcB: (\gamma, \gamma')\mapsto \gamma\gamma':=\gamma\cup \gamma'.
$$
By identifying $1\in \bQ$ with $1\in \hcB$, this gives a natural graded $\bQ$-algebra structure of $\hcB$.
\end{definition}

\begin{definition}[Differential operator]\label{def partial}
There is a differential operator $\partial: \hcB'\to \hcB'$ defined by
$$
\partial \Gamma=\sum_{\{u,v\}\subset U}\Gamma/\{u,v\}
$$
for every Jacobi diagram $\Gamma$.
Here $U$ is the set of univalent vertices of $\Gamma$, and $\Gamma/\{u,v\}$ is the Jacobi diagram obtaining by removing vertices $\{u, v\}$ and gluing $2$ edges belonging to $u,v$ into a new edge. Here $\Gamma/\{u,v\}$ admits a natural orientation from $\Gamma$ as trivalent vertices remain unchanged.
In other words, the action of $\partial$ on a Jacobi diagram means to glue $2$ of its univalent vertices in all possible ways. Note that $\partial: \hcB'\to \hcB'$ is a $\hcB^t$-linear map.
\end{definition}

\begin{example}
$\partial \bw_2=\Theta.$
\end{example}

\begin{definition}[Wheeling element]
The {\it wheeling element} $\Omega\in \hcB$ is defined
via the expression
$$
\Omega=\exp\left(\sum_{k=1}^{\infty}b_{2k}\bw_{2k}\right)
$$
using the graded $\bQ$-algebra structure of $\hcB$, where $b_{2k}$ are the modified Bernoulli numbers as in Section~\ref{sec todd}.
We may write $\Omega=\sum_{k=0}^{\infty}\Omega_{2k}$ where $\Omega_{2k}$ is the homogeneous
component of degree $4k$ of $\Omega$. For example, $\Omega_0=1$, $\Omega_2=\frac{1}{48}\bw_2$.
\end{definition}

The wheeling element $\Omega$ has an important property called the wheeling theorem (see \cite{Thurston}) by the knot theory.
The method of combining the wheeling theorem with Rozansky--Witten classes to deal with characteristic classes of hyperk\"{a}hler manifolds was discovered by Hitchin and Sawon \cite{hitchinsawon} and later generalized by Nieper-Wi{\ss}kirchen \cite{nieper-jag}.
As observed by Nieper-Wi{\ss}kirchen, all we need is the following special case. 
\begin{thm}[{\cite[Lemma 6.2]{Thurston}, \cite[Theorem 3.1]{nieper-jag}}]\label{wheeling thm}
As elements in $\hcB$, $\partial \Omega=\frac{\Theta}{48}\Omega$. 
\end{thm}
An elementary proof can be found in \cite{nieper-jag}.

\subsection{Rozansky--Witten classes in general setting}

Let $\bf{k}$ be a field of characteristic zero, $V$ a finite-dimensional $\bf{k}$-vector space, $A=\bigoplus_{i=0}^{\infty} A_i$ a skew-commutative $\mathbb{Z}$-graded $\bf{k}$-algebra, and $\sigma$ a symplectic form on $V$.
We will apply this general setting later to the case that $V=\mathcal{T}_{X, x}$ the holomorphic tangent space and $A=\bigwedge \overline{\Omega}_{X, x}$, where $x\in X$ is a point on a hyperk\"{a}hler manifold $X$.

For every Jacobi diagram $\Gamma$ with $k$ trivalent and $l$ univalent vertices and every $\alpha\in \Sym^3V \otimes A_1$, we define an element
$$
\RW_{\sigma, \alpha}(\Gamma)\in {\bigwedge}^lV^*\otimes A_k
$$
as the following (see \cite{Kapranov}, \cite[Section 4.1]{nieper-jag}).

Denote $T$ to be the set of trivalent vertices of $\Gamma$, denote $U$ to be the set of univalent vertices of $\Gamma$, and denote $F$ to be the set of flags of $\Gamma$. 
Recall that an edge is identified with $2$ flags belonging to it, and a vertex is identified with the set of its flags.
We can label the set $F=\{1,2,\dots, 3k+l-1, 3k+l\}$ such that the set of edges are just $E=\{\{1,2\},\dots, \{3k+l-1, 3k+l\}\}$. Identifying vertices with sets of flags, we may write the set of univalent vertices as $U=\{u_1, \dots, u_l\}\subset F$, and write the set of trivalent vertices as
$$
T=\{\{t_1,t_2, t_3\}, \cdots, \{t_{3k-2}, t_{3k-1}, t_{3k}\}\}\subset 2^F,
$$
where the ordering $\{t_{3i-2}, t_{3i-1}, t_{3i}\}$ coincides with the orientation of $\Gamma$ (which is a given cyclic ordering for each trivalent vertex).
Note that we have the relation
$$\coprod_{t\in T}t\cup \coprod_{u\in U} u=\coprod_{e\in E}e=F.$$
Divide $U=U'\cup U''$ where $U'=\{u'_{1},\dots, u'_{l_1}\}$ consists of univalent vertices connected to 
trivalent vertices, and $U''=\{u''_{1},\dots, u''_{2l_2}\}$ consists of univalent vertices contained in some component $\ell$ of $\Gamma$. Here $l=l_1+2l_2$ and $\Gamma$ has exactly $l_2$ copies of $\ell$ as connected components.

We may assume that for $1\leq j\leq l_2$, $(u''_{2j-1}, u''_{2j})=(3k+l_1+2j-1, 3k+l_1+2j)$, 
that is, in the above ordering, the last $l_2$ edges correspond to the components $\ell^{l_2}\subset \Gamma$. 
Note that for $1\leq i\leq l_1$, we have $1\leq u'_{i}\leq 3k+l_1$, so we may further assume that each $u'_{i}$ is even, that is, the flag belonging to it takes the second position in the ordering on the corresponding edge, which is just $\{u'_{i}-1, u'_{i}\}$. 

We may choose the labelling of $F$ properly such that 
\begin{align}
{}&
\bigwedge_{\substack{1\leq j\leq 3k+l_1\\ j\neq u'_1,\dots, u'_{l_1}}} f_j
=f_{t_1}\wedge\dots\wedge f_{t_{3k}}.\label{orientation} 
\end{align}
in ${\bigwedge} ({\bf k}^{\oplus 3k})$ where $\{f_{t_i}\}_{i=1}^{3k}$ is a basis of ${\bf k}^{\oplus 3k}$, and this condition is called the {\it compatibility with the orientation of $\Gamma$}. See Figure~\ref{figure 1} for an example of $\bw_2\cup \ell$ with $k=2$, $l=4$. 
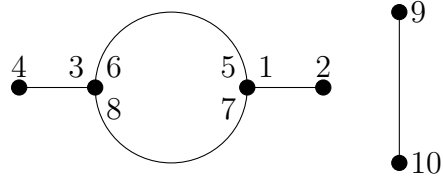
\begin{figure}\caption{A labelling of $\bw_2\cup \ell$}\label{figure 1}
\begin{tikzpicture} 
\draw (1,0) node[above left] {5} node[below left] {7} node[above right] {1}-- (2,0) node[above] {2}; 
\draw (-2,0) node[above] {4} -- (-1,0) node[above left] {3} node[above right] {6} node[below right] {8} ; 
\draw (0,0) circle (1); 
\draw (3, -1) node[right] {10} -- (3,1)node[right] {9} ;
\filldraw (-1,0) circle (.1)
(1,0) circle (.1)
(-2,0) circle (.1)
(2,0) circle (.1)
(3,1) circle (.1)
(3,-1) circle (.1);
\end{tikzpicture} 
\end{figure}
\begin{remark}\label{remark orientation condition}
\begin{enumerate}
\item If we ignore the compatibility of orientation, then the Rozansky--Witten invariants are only defined up to a sign.
\item The compatibility condition \eqref{orientation} we state here is different from the one in \cite[(3.3)]{nieper-jag} up to a sign $(-1)^{l_1(l_1-1)/2}$ because we ignore the contribution from univalent vertices. 
In fact, the referee reminded the author that in \cite[(4.2)]{nieper-jag}, each isomorphism $V\otimes A_1 \to A_1 \otimes V$ introduces a sign change, and so the definition in this paper is indeed the same as that in \cite{nieper-jag}.

\item If we glue $u'_s$ and $u'_t$ in $U'$ ($1\leq u'_s<u'_t\leq 3k+l_1$), then we get $\Gamma_{s,t}:=\Gamma/\{u'_s, u'_s\}$ as in Definition~\ref{def partial}.
By the assumption, the opposite flags corresponding to the edges containing $u'_s, u'_t$ are $u'_s-1, u'_t -1$ respectively. 
Note that the set of flags $F_{s,t}$ (resp. the set of univalent vertices $U_{s,t}$) of $\Gamma_{s,t}$ is $F$ (resp. $U$) removing $u'_s, u'_t$, and the set of edges $E_{s,t}$ of $\Gamma_{s,t}$ is $E$ removing $\{u'_s-1, u'_s\}, \{u'_t-1, u'_t\}$ and adding $\{u'_s-1, u'_t-1\}$ as a new edge after the edge $\{3k+l_1-1, 3k+l_1\}$. Then the ordering of $F_{s,t}$ induced from $E_{s,t}$ satisfies the compatibility with the orientation of $\Gamma_{s,t}$ up to a sign $(-1)^{s+t-1}$. This is because from $\eqref{orientation}$, we have
\begin{align*}
\bigwedge_{\substack{1\leq j\leq 3k+l_1\\ j\neq u'_1,\dots, u'_{l_1}\\ j\neq u'_s-1, u'_t-1}} f_j \wedge (f_{u'_s-1}\wedge f_{u'_t-1})
=(-1)^{s+t-1} f_{t_1}\wedge\dots\wedge f_{t_{3k}}.
\end{align*}
For example, if we glue $2,4$ in Figure~\ref{figure 1}, we get a labeling of flags $\{5,6,7,8,1,3,9,10\}$ on $\Theta\cup \ell$ compatible with the natural orientation, as shown in Figure~\ref{figure 2}.
\begin{figure}\caption{A labelling of $\Theta\cup \ell$}\label{figure 2}
\begin{tikzpicture} 
\draw (-1,0) node[above left] {3} node[below left] {8} node[ right] {6}-- (1,0) node[above right] {1} node[ left] {5} node[below right] {7} ; 
\draw (0,0) circle (1); 
\draw (3, -1) node[right] {10} -- (3,1)node[right] {9} ;
\filldraw (-1,0) circle (.1)
(1,0) circle (.1)
(3,1) circle (.1)
(3,-1) circle (.1);
\end{tikzpicture} 
\end{figure}
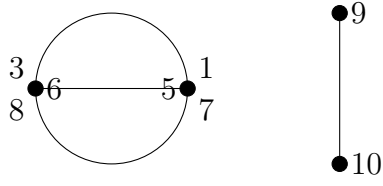
\end{enumerate}
\end{remark}

Note that we may identify $\text{\rm End}(V)\simeq V\otimes V^*$. Consider an element in $(\Sym^3V \otimes A_1)^{\otimes |T|}\otimes \text{\rm End}(V)^{\otimes |U|}$
of the form 
$$s=\bigotimes_{i=1}^k(w_{i}^3\otimes a_i)\otimes \bigotimes_{j=1}^{l}(w'_j\otimes \alpha_j),$$
where $w_{i}\in V$, $a_i\in A_1$, $w'_j\in V$, and $\alpha_j\in V^*$ for each $i, j$.
For each $i, j$, we denote $v_{t_{3i-1}}=v_{t_{3i-2}}=v_{t_{3i}}=w_i$ and $v_{u_j}=w'_j$. 
Then we define 
$$\Phi^{\Gamma}(s):=\prod_{j=1}^{(3k+l)/2}\sigma(v_{2j-1},v_{2j})\cdot(\alpha_1\wedge\dots\wedge\alpha_l)\otimes (a_1\cdot\dots\cdot a_k).$$
Roughly speaking, $s$ assigns to every flag an element in $V$, and $\Phi^{\Gamma}$ contracts every two elements on the same edge by $\sigma$ in a way compatible with the orientation of $\Gamma$.
This in fact extends to a linear map 
$$\Phi^{\Gamma}: (\Sym^3V \otimes A_1)^{\otimes |T|}\otimes \text{\rm End}(V)^{\otimes |U|}\to {\bigwedge}^lV^*\otimes A_k. $$
This definition can be extended to $\Phi^{\gamma}$ for any $\mathbb{Q}$-linear combinations of Jacobi diagrams with $k$ trivalent and $l$ univalent vertices $\gamma$ by linearity.

\begin{definition}
With the above notation, for every Jacobi diagram $\Gamma$ with $k$ trivalent and $l$ univalent vertices and every $\alpha\in \Sym^3V \otimes A_1$, we define the element
$$
\RW_{\sigma, \alpha}(\Gamma):=\Phi^{\Gamma}(\alpha^{\otimes k}\otimes (\text{\rm id}_V)^{\otimes l})\in {\bigwedge}^lV^*\otimes A_k.
$$
This definition can be extended to $\RW_{\sigma, \alpha}(\gamma)\in {\bigwedge}V^*\otimes A$ for any $\mathbb{Q}$-linear combinations of Jacobi diagrams $\gamma$ by linearity.
\end{definition}


\begin{example}[{cf. \cite[Proposition 3, (4.5)]{nieper-jag}}]\label{example Phi w2k}
Let us compute $\Phi^{\bw_{2k}}.$
We fix the labeling of flags as the following:
the $2k$ univalent vertices are labelled as $U=\{2,4,6,\dots, 4k\}$ counter-clockwisely, the 
$2k$ trivalent vertices are labelled as 
\begin{align*}T=\{\{1,4k+1, 8k\}, \{3, 4k+3, 4k+2\}, {}&\{5, 4k+5, 4k+4\}, \dots,\\
{}& \{4k-1, 8k-1, 8k-2\}\}.
\end{align*}
Here $\{2i-1, 2i\}$ is the edge from a trivalent vertex to a univalent vertex for $1\leq i\leq 2k$.
Then the set of edges is just
$$
E=\{\{1,2\},\dots, \{8k-1, 8k\}\}.
$$
See Figure~\ref{figure 3}.
\begin{figure}\caption{A labelling of $\bw_{2k}$}\label{figure 3}
\begin{tikzpicture} [scale=1]
\tikzstyle{every node}=[font=\small,scale=0.5]
\draw (1,0) node[above left] {$4k+1$} node[below left] {$8k$} node[above right] {1}-- (1.5,0) node[above] {2}; 
\draw (0,1) node[above left] {3} node[below left] {$4k+3$} node[below right] {$4k+2$}-- (0,1.5) node[left] {4}; 
\draw (-1.5,0) node[above] {6} -- (-1,0) node[above left] {5} node[above right] {$4k+4$} node[below right] {$4k+5$} ; 
\draw (-0.7071,-0.7071) node[above right] {$\dots$} -- (-1.05,-1.05); 
\draw (0,-1) node[above left] {$8k-2$} node[above right] {$8k-1$} node[below right] {$4k-1$}-- (0,-1.5) node[right] {$4k$}; 
\draw (0,0) circle (1); 
\end{tikzpicture} 
\end{figure}
Note that this given labeling is compatible with the orientation of $\bw_{2k}$ up to a sign $(-1)$, as 
\begin{align}
{}&f_1\wedge f_3\wedge\dots \wedge f_{4k-1}\wedge f_{4k+1}\wedge\dots \wedge f_{8k}\label{compatible w2k}\\
={}&-(f_{1}\wedge f_{4k+1}\wedge f_{8k})\wedge \dots \wedge (f_{4k-1}\wedge f_{8k-1}\wedge f_{8k-2})\notag
 \end{align}
 in ${\bigwedge} ({\bf k}^{\oplus 6k})$ where $\{f_{i}\}$ is a basis of ${\bf k}^{\oplus 6k}$ (cf. \cite[Page 36]{Sawon-phd}). One can compare this labelling with the labelling of $\bw_2$ in Figure~\ref{figure 1} to feel the difference.

Then for an element in $(\Sym^3V \otimes A_1)^{\otimes 2k}\otimes \text{\rm End}(V)^{\otimes 2k}$
of the form 
$$s=\bigotimes_{i=1}^{2k}(w_i^3\otimes a_i)\otimes \bigotimes_{j=1}^{2k}(w'_j\otimes \alpha_j),$$
where $w_i\in V$, $a_i\in A_1$, $w'_j\in V$, and $\alpha_j\in V^*$ for each $1\leq i, j\leq 2k$,
we have
$$\Phi^{\bw_{2k}}(s)=-\prod_{i=1}^{2k}\bigg(\sigma(w_i,w_{i+1})\sigma(w_i,w'_{i})\bigg)\cdot\bigwedge_{i=1}^{2k}\alpha_i\otimes \prod_{i=1}^{2k}a_i.$$
Here we set $w_{2k+1}:=w_1$, and the $(-1)$ sign is from the compatibility with the orientation, i.e., Equation~\eqref{compatible w2k}.
\end{example}
Properties of the map $\Phi^{\Gamma}$ are summarized in \cite[Proposition 3]{nieper-jag}. Here we introduce a property originally observed by Nieper-Wi{\ss}kirchen \cite[(85)]{nieper-phd}.

\begin{prop}[{cf. \cite[(85)]{nieper-phd}}]\label{RW diff local}
Keep the above notation. Fix a Jacobi diagram $\Gamma$ with $k$ trivalent and $l$ univalent vertices, suppose moreover that $[\Gamma]\in \cB'$. 
Then for any $\beta\in (\Sym^3V \otimes A_1)^{\otimes k}$, 
$$\Phi^{\partial \Gamma}(\beta\otimes (\text{\rm id}_V)^{\otimes (l-2)})=\delta(\Phi^{\Gamma}(\beta\otimes (\text{\rm id}_V)^{\otimes l})).$$
In particular, $\RW_{\sigma, \alpha}(\partial \Gamma)=\delta(\RW_{\sigma,\alpha}(\Gamma))$ for any $\alpha\in \Sym^3V \otimes A_1$. Here $\delta$ is the contraction map in Definition~\ref{def contraction}.
\end{prop}

\begin{proof}
As in the construction, we may write the set of flags $F$ of $\Gamma$ by $\{1,2,\dots, 3k+l-1, 3k+l\}$
such that the set of edges $E$ are just $\{\{1,2\},\dots, \{3k+l-1, 3k+l\}\}$. Identifying vertices with sets of flags, we may write the set of univalent vertices $U$ as $\{u_1, \dots, u_l\}\subset F$, and write the set of trivalent vertices $T$ as
$$
\{\{t_1,t_2, t_3\}, \cdots, \{t_{3k-2}, t_{3k-1}, t_{3k}\}\}\subset 2^F,
$$
where the ordering $\{t_{3i-2}, t_{3i-1}, t_{3i}\}$ coincides with the orientation of $\Gamma$.
Note that $U=U'$ and $l=l_1$ by the assumption that $[\Gamma]\in \cB'$.
Without loss of generality, we may assume that for $1\leq i\leq l$, $(t_{3i-2}, u_i)=(2i-1, 2i)$, that is, in the above ordering, the $i$-th univalent vertex is connected to the $i$-th trivalent vertex via the $i$-th edge. 

Without loss of generality, we may assume that the element $\beta\in (\Sym^3V \otimes A_1)^{\otimes k}$ is of the form 
$$
\beta=\bigotimes_{i=1}^k(w_i^3\otimes a_i)
$$
where $v_{t_{3i-1}}=v_{t_{3i-2}}=v_{t_{3i}}=w_i\in V$ and $a_i\in A_1$ for each $i$.
Choose a symplectic basis $e_1,\dots, e_{2n}$ of $V$ such that $$\sigma=\sum_{i=1}^n\vartheta^{2i-1}\wedge \vartheta^{2i},$$ where $\vartheta^1,\dots, \vartheta^{2n}$ is the corresponding dual basis of $V^*$.
Then $\text{\rm id}_V=\sum_{m=1}^{2n}e_m\otimes \vartheta^m$ via $\text{\rm End}(V)\simeq V\otimes V^*$.
Then
\begin{align*}
{}&\Phi^{\Gamma}(\beta\otimes (\text{\rm id}_V)^{\otimes l})\\
={}&\Phi^{\Gamma}\left(\beta\otimes \bigg(\sum_{m=1}^{2n}e_m\otimes \vartheta^m\bigg)^{\otimes l}\right)\\
={}&\sum_{1\leq m_1,\dots,m_l\leq 2n}
\left(
\begin{aligned}
{}&\prod_{i=1}^{l}\sigma(v_{2i-1},e_{m_{i}})\cdot \prod_{i'=l+1}^{(3k+l)/2} \sigma(v_{2i'-1},v_{2i'})\\
{}&\cdot(\vartheta^{m_1}\wedge\dots\wedge\vartheta^{m_l})\otimes (a_1\cdot\dots\cdot a_k)
\end{aligned}
\right)\\
={}&{\bigwedge}_{i=1}^l\left(\sum_{m=1}^{2n}\sigma(v_{2i-1},e_{m}) \vartheta^m\right)\otimes \left(\prod_{i'=l+1}^{(3k+l)/2} \sigma(v_{2i'-1},v_{2i'})\cdot (a_1\cdot\dots\cdot a_k)\right)\\
={}&{\bigwedge}_{i=1}^l\left(\sum_{m=1}^{2n}\sigma(v_{2i-1},e_{m}) \vartheta^m\right)\otimes N.
\end{align*}
Here we denoted $N= \prod_{i'=l+1}^{(3k+l)/2} \sigma(v_{2i'-1},v_{2i'})\cdot (a_1\cdot\dots\cdot a_k)$.
Then 
{\tiny
\begin{align*}
{}&\delta(\Phi^{\Gamma}(\beta\otimes (\text{\rm id}_V)^{\otimes l}))\\
={}&\sum_{r=1}^{n}\sum_{1\leq s<t\leq l}\left(
\begin{aligned}
{}&(-1)^{s+t-1}\bigg(\sigma(v_{2s-1}, e_{2r-1})\sigma(v_{2t-1}, e_{2r})-\sigma(v_{2s-1}, e_{2r})\sigma(v_{2t-1}, e_{2r-1})\bigg)\\
{}&\cdot {\bigwedge}_{\substack{1\leq i\leq l\\i\neq s, t}}\left(\sum_{m=1}^{2n}\sigma(v_{2i-1},e_{m})\vartheta^m\right)
\end{aligned}\right)
\otimes N\\
={}&\sum_{1\leq s<t\leq l}
\left((-1)^{s+t-1}\sigma(v_{2s-1}, v_{2t-1}) \cdot {\bigwedge}_{\substack{1\leq i\leq l\\i\neq s, t}}\left(\sum_{m=1}^{2n}\sigma(v_{2i-1},e_{m})\vartheta^m\right)
\right)
\otimes N.
\end{align*}
}
Here we used the fact that 
\begin{align*}
{}&\sigma(v_{2s-1}, v_{2t-1})\\
={}& \sum_{r=1}^n\bigg(\sigma(v_{2s-1}, e_{2r-1})\sigma(v_{2t-1}, e_{2r})-\sigma(v_{2s-1}, e_{2r})\sigma(v_{2t-1}, e_{2r-1})\bigg).
\end{align*}

On the other hand, denote $\Gamma_{s,t}=\Gamma/\{u_s, u_t\}$ for ${\{u_s, u_t\}\subset U}$, note that 
$$
\partial \Gamma=\sum_{s<t}\Gamma_{s,t}.
$$
Here as the above notation, $(u_s, u_t)=(2s, 2t)$.
Now we compute $\Phi^{ \Gamma_{s,t}}(\beta\otimes (\text{\rm id}_V)^{\otimes (l-2)})$.
Note that the set of flags (resp. univalent vertices) of $\Gamma_{s,t}$ is $F$ (resp. $U$) removing $2s, 2t$, the set of edges of $\Gamma_{s,t}$ is $E$ removing $\{2s-1, 2s\}, \{2t-1, 2t\}$ and adding $\{2s-1, 2t-1\}$ as a new edge to the end.
So we may compute
{\tiny
\begin{align*}
{}&\Phi^{ \Gamma_{s,t}}(\beta\otimes (\text{\rm id}_V)^{\otimes (l-2)})\\
={}&\Phi^{ \Gamma_{s,t}}\left(\beta\otimes \bigg(\sum_{m=1}^{2n}e_m\otimes \vartheta^m \bigg)^{\otimes (l-2)}\right)\\
={}&\sum_{\substack{1\leq m_j\leq 2n\\1\leq j\leq l\\j\neq s,t}}
\left(
\begin{aligned}
{}&(-1)^{s+t-1} \prod_{\substack{1\leq i\leq l\\ i\neq s,t }}\sigma(v_{2i-1},e_{m_{i}})\cdot \prod_{i'=l+1}^{(3k+l)/2} \sigma(v_{2i'-1},v_{2i'})\cdot\sigma(v_{2s-1}, v_{2t-1})\\
{}&\cdot(\vartheta^{m_1}\wedge\dots\wedge\widehat{\vartheta^{m_s}}\wedge\dots\wedge\widehat{\vartheta^{m_t}}\wedge\dots \wedge\vartheta^{m_l})\otimes (a_1\cdot\dots\cdot a_k)
\end{aligned}
\right)\\
={}&(-1)^{s+t-1}\sigma(v_{2s-1}, v_{2t-1})\cdot{\bigwedge}_{\substack{1\leq i\leq l\\ i\neq s,t }}\left(\sum_{m=1}^{2n}\sigma(v_{2i-1},e_{m}) \vartheta^m\right)\otimes N.
\end{align*}}
Here the sign $(-1)^{s+t-1}$ comes from the compatibility of the orientation of $\Gamma_{s,t}$ (Remark~\ref{remark orientation condition}(3)).

To conclude, we get 
\begin{align*}\Phi^{\partial \Gamma}(\beta\otimes (\text{\rm id}_V)^{\otimes (l-2)})={}&\sum_{1\leq s<t\leq l}\Phi^{\Gamma_{s,t}}(\beta\otimes (\text{\rm id}_V)^{\otimes (l-2)})\\
={}&\delta(\Phi^{\Gamma}(\beta\otimes (\text{\rm id}_V)^{\otimes l}))
\end{align*}
by comparing the above computations.
\end{proof}

\subsection{Rozansky--Witten classes of hyperk\"{a}hler manifolds}
Let $X$ be a hyperk\"{a}hler manifold and fix a non-zero $\sigma\in H^0(X, \Omega^2_X)$.
Denote $A^k(X, E)$ to be the
space of $(0,k)$-forms with values in a holomorphic vector bundle $E$, and set $A^{l, k}(X) : = A^k (X, \Omega_X^l)$. Denote $\alpha_X$ to be a
Dolbeault representative of the {\it Atiyah class} of $X$. Recall that according to Kapranov \cite{Kapranov}, $\alpha_X\in A^1(X,\Sym^3 \mathcal{T}_X)$ when we identify $\mathcal{T}_X$ with $\Omega_X$ via $\sigma$. 

\begin{definition}[{\cite[Definition 14]{nieper-jag}}]
For a Jacobi diagram $\Gamma$ with $k$ trivalent and $l$ univalent vertices, we define
$\RW_{\sigma}(\Gamma) \in H^k(X,\Omega_X^l)$
to be the Dolbeault cohomology class of the ($\bar\partial$)-closed $(l,k)$-form $$\left(\frac{\sqrt{-1}}{2\pi}\right)^k\cdot (x\mapsto \RW_{\sigma_x, \alpha_{X,x}} (\Gamma)) \in A^{l,k} (X).$$
This definition can be extended to $\RW_{\sigma}(\gamma)$ for any $\mathbb{Q}$-linear combinations of Jacobi diagrams $\gamma$ by linearity.
\end{definition}

It was proved by Kapranov \cite{Kapranov} that $\RW_{\sigma}(\Gamma)=\RW_{\sigma}(\Gamma')$ if $[\Gamma]=[\Gamma']\in \hcB$. So we have a linear map $\RW_\sigma: \hcB\to H^{*}(X, \Omega_X^*)$ which maps 
 elements of $\hcB_{k,l}$ into $H^{k}(X, \Omega_X^l)$. The values of $\RW_\sigma$ are called the {\it Rozansky--Witten classes} of $X$. In fact, this map preserves $\mathbb{Q}$-algebra structures.
 
 \begin{prop}[{\cite[Proposition 6]{nieper-jag}}]\label{prop RW algebra}
 $\RW_\sigma: \hcB\to H^{*}(X, \Omega_X^*)$ is a morphism of $\mathbb{Q}$-algebras.
 \end{prop}

For special Jacobi diagrams, we have the following corresponding Rozansky--Witten classes.
\begin{prop}[{\cite{nieper-jag}}]\label{prop RW classes}Let $X$ be a hyperk\"{a}hler manifold and fix a non-zero $\sigma\in H^0(X, \Omega^2_X)$.
Then
\begin{enumerate}
\item $\RW_\sigma(\ell)=2\sigma$;
\item $\RW_\sigma(\Theta)=\frac{48}{\lambda_\sigma}\sigmabar$, where $\lambda_\sigma=\lambda(\sigma+\sigmabar)>0$ as in Definition~\ref{def lambda}; 
\item $\RW_{\sigma}(\bw_{2k})=-(2k)!\ch_{2k}(X);$
\item $\RW_{\sigma}(\Omega)=\td^{1/2}(X)$.
\end{enumerate}
\end{prop}

\begin{proof}
(1)--(3) follow from \cite[(4.11), (4.16), (4.12)]{nieper-jag} (cf. Remark~\ref{remark orientation condition}(2)). (4) follows from (3) and Proposition~\ref{prop RW algebra} (cf. \cite[(4.14)]{nieper-jag}).
\end{proof}
\subsection{Conclusions}
As a direct application of Proposition~\ref{RW diff local} and Remark~\ref{remark local sl2}, we have the following theorem originally observed by Nieper-Wi{\ss}kirchen \cite[(85)]{nieper-phd}.

\begin{thm}[{Nieper-Wi{\ss}kirchen \cite[(85)]{nieper-phd}}]\label{RW diff}
Let $X$ be a hyperk\"{a}hler manifold and fix a non-zero $\sigma\in H^0(X, \Omega^2_X)$.
Then for any $\gamma\in \hcB'$,
$\RW_{\sigma}(\partial \gamma)=\Lambda_{\sigma/4}(\RW_\sigma(\gamma)).$
\end{thm}

\begin{example}\label{example k3}
Let $X$ be a K3 surface and fix a non-zero $\sigma\in H^0(X, \Omega^2_X)$. We may assume that $\int \sigma\sigmabar=1$. Then we have $c_2(X)=24\sigma\sigmabar$ and $\lambda_\sigma=\lambda(\sigma+\sigmabar)=1$.
 Consider $\gamma=\bw_2$. 
Then $$\RW_{\sigma}(\partial \bw_2)=\RW_{\sigma}(\Theta)=\frac{48}{\lambda_\sigma}\sigmabar=48\sigmabar.$$ On the other hand,
$$\Lambda_{\sigma/4}(\RW_\sigma(\bw_2))=\Lambda_{\sigma/4}(2c_2(X))=\Lambda_{\sigma/4}(48\sigma\sigmabar)=48\sigmabar.$$
Hence $\RW_{\sigma}(\partial \bw_2)=\Lambda_{\sigma/4}(\RW_\sigma(\bw_2)).$
\end{example}

\begin{remark}\label{remark nieper sign}
Here we make a historical remark on Theorem~\ref{RW diff}. 
In \cite[(85)]{nieper-phd}, Nieper-Wi{\ss}kirchen claimed a formula as in Theorem~\ref{RW diff} with a different sign. The reason is that his definition of $\delta$ differs from ours (Definition~\ref{def contraction}) by a sign.
In fact, Definition~\ref{def contraction} coincides with \cite[Definition~29]{nieper-phd}, but differs from \cite[Remark~30]{nieper-phd} by a sign. So to clarify which sign is the right choice, we decide to include a detailed proof of Proposition~\ref{RW diff local} and Theorem~\ref{RW diff} in this paper following our sign conventions. In any sense these two results should be credited to Nieper-Wi{\ss}kirchen who observed this interesting correspondence. 
\end{remark}

Combining Theorem~\ref{RW diff} with Theorem~\ref{wheeling thm}, we get the following consequence, which is crucial in the study of $\td^{1/2}(X)$.

\begin{cor}\label{lambda td}Let $X$ be a hyperk\"{a}hler manifold and fix a non-zero $\sigma\in H^0(X, \Omega^2_X)$.
Then for any integer $k\geq 1$,
$$\Lambda_{\sigma/4}(\td^{1/2}_{2k})=\frac{1}{\lambda_\sigma}\td^{1/2}_{2k-2}\wedge\sigmabar.$$
\end{cor}

\begin{proof}
By Proposition~\ref{prop RW classes}(4), $\td^{1/2}_{2k}=\RW_\sigma(\Omega_{2k}).$ By Theorem~\ref{wheeling thm}, we have $\partial \Omega_{2k}=\frac{\Theta}{48}\Omega_{2k-2}$ by taking the homogenous parts of degree $4k-2$.
So by Theorem~\ref{RW diff},
\begin{align*}
\Lambda_{\sigma/4}(\td^{1/2}_{2k})={}&\Lambda_{\sigma/4}(\RW_\sigma(\Omega_{2k}))=\RW_\sigma(\partial\Omega_{2k})\\
={}&\RW_\sigma\left(\frac{\Theta}{48}\Omega_{2k-2}\right)=\RW_\sigma\left(\frac{\Theta}{48}\right)\RW_\sigma(\Omega_{2k-2})\\
={}&\frac{1}{\lambda_\sigma}\td^{1/2}_{2k-2}\wedge\sigmabar.
\end{align*}
Here we applied Proposition~\ref{prop RW algebra} and Proposition~\ref{prop RW classes}(2).
\end{proof}

Corollary~\ref{lambda td} can be also written as the following with respect to K\"{a}hler forms.
\begin{cor}\label{lambda td2}Let $X$ be a hyperk\"{a}hler manifold with a K\"{a}hler form $\omega$. Then for any integer $k\geq 1$,
$$\Lambda_{\omega}(\td^{1/2}_{2k})=\frac{1}{\lambda(\omega)}\td^{1/2}_{2k-2}\wedge \omega.$$
\end{cor}
\begin{proof}
Consider $\sigma=\omega_J+\sqrt{-1}\omega_K\in H^0(X, \Omega^2_X)$ as in Section~\ref{sec IJK}. Then Corollary~\ref{lambda td} says that
$$\frac{1}{4}(\Lambda_{\omega_J}-\sqrt{-1}\Lambda_{\omega_K})(\td^{1/2}_{2k})=\frac{1}{\lambda_\sigma}\td^{1/2}_{2k-2}\wedge(\omega_J-\sqrt{-1}\omega_K).$$
As $\td^{1/2}_{2k}$ and $\td^{1/2}_{2k-2}$ are real, we get
$$
\Lambda_{\omega_J}(\td^{1/2}_{2k})=\frac{4}{\lambda_\sigma}\td^{1/2}_{2k-2}\wedge \omega_J=\frac{1}{\lambda(\omega_J)}\td^{1/2}_{2k-2}\wedge \omega_J.
$$
Here we used the fact that $\lambda(\omega_J)=\lambda(\frac{1}{2}(\sigma+\sigmabar))=\frac{1}{4}\lambda_\sigma.$
Hence the conclusion also holds by replacing $\omega_J$ with $\omega$ according to the hyperk\"{a}hler structure.
\end{proof}

\begin{remark}\label{rem 321}
\begin{enumerate}
\item As $\td^{1/2}_{2k}$ and $\td^{1/2}_{2k-2}$ are real, by conjugation,
Corollary~\ref{lambda td} also implies that 
$$\Lambda_{\sigmabar/4}(\td^{1/2}_{2k})=\frac{1}{\lambda_\sigma}\td^{1/2}_{2k-2}\wedge\sigma,$$
where $\Lambda_{\sigmabar/4}$ can be defined similarly as $\Lambda_{\sigma/4}$ as in Section~\ref{sec sl2}.
So Corollaries~\ref{lambda td} and~\ref{lambda td2} show that $\td^{1/2}_{2k-2}$ and $\td^{1/2}_{2k}$ coincide after taking proper Lefschetz operations, as shown by the following diagram:
$$\xymatrix{& \td^{1/2}_{2k-2}\in H^{2k-2}(X, \Omega^{2k-2}_{X})\ar[d]_{\frac{1}{\lambda(\omega)}L_{\omega}} \ar[dr]^{\frac{1}{\lambda_\sigma}L_{\sigma}} \ar[dl]_{\frac{1}{\lambda_\sigma}L_{\sigmabar}} &\\
H^{2k}(X, \Omega^{2k-2}_{X}) & H^{2k-1}(X, \Omega^{2k-1}_{X}) & H^{2k-2}(X, \Omega^{2k}_{X}).\\
& \td^{1/2}_{2k}\in H^{2k}(X, \Omega^{2k}_{X}) \ar[u]^{\Lambda_{\omega}} \ar[ul]^{\Lambda_{\sigma/4}} \ar[ur]_{\Lambda_{\sigmabar/4}} &}
$$
\item Note that by the Hodge theory, $L_\omega$ is injective on $H^{2k}(X, \Omega^{2k}_{X})$ for $k\leq \frac{n}{2}$ and $\Lambda_\omega$ is injective on $H^{2k}(X, \Omega^{2k}_{X})$ for $k> \frac{n}{2}$. So by Corollary~\ref{lambda td2}, $\td^{1/2}_{2\rounddown{n/2}}$ carries all information of $\{\td^{1/2}_{2k}\mid 0\leq k\leq n\}$.
\end{enumerate}
\end{remark}
\section{A Lefschetz-type decomposition of $\td^{1/2}(X)$}\label{sec lef dec}
In this section, we apply Corollary~\ref{lambda td} to give a 
Lefschetz-type decomposition of $\td^{1/2}(X)$. 
Note that we can also apply Corollary~\ref{lambda td2} to give a 
similar Lefschetz-type decomposition of $\td^{1/2}(X)$ (see Remark~\ref{primitive omega}), but the former one is easier to handle in computations. 

 First, we find several natural primitive elements given by linear combinations of $\td^{1/2}_{2k}$.

\begin{definition}Let $X$ be a hyperk\"{a}hler manifold of dimension $2n$ and fix a non-zero $\sigma\in H^0(X, \Omega^2_X)$. Consider $\lambda_\sigma=\lambda(\sigma+\sigmabar)>0$ as in Definition~\ref{def lambda}.
For $0\leq k\leq n/2$, denote
$$
\tp_{2k}:=\sum_{i=0}^k\frac{(n-2k+1)!\td^{1/2}_{2i}\wedge(\sigma\sigmabar)^{k-i}}{(-\lambda_\sigma)^{k-i}(k-i)!(n-k-i+1)!}\in H^{4k}(X).$$
Note that it is of type $(2k,2k)$ with respect to the usual Hodge decomposition. In particular, $\tp_0=1.$
\end{definition}

The following proposition shows that $\tp_{2k}$ is indeed primitive in several senses.

\begin{prop}\label{prop primitive}Let $X$ be a hyperk\"{a}hler manifold of dimension $2n$ and fix a non-zero $\sigma\in H^0(X, \Omega^2_X)$.
Then for any $0\leq k\leq n/2$, $\tp_{2k}$ is both $\sigma$-primitive and $(\sigma+\sigmabar)$-primitive.
\end{prop}
\begin{proof}
To show that $\tp_{2k}$ is $\sigma$-primitive, we need to check that $\Lambda_{\sigma/4}(\tp_{2k})=0.$ This can be checked directly by using Corollary~\ref{lambda td}. In fact, by Corollary~\ref{lambda td}, Lemmas~\ref{sigma sigmabar commutes} and 
\ref{AL=L},
\begin{align*}
{}&\Lambda_{\sigma/4}(\tp_{2k})/(n-2k+1)!\\
={}&\sum_{i=0}^k\frac{\Lambda_{\sigma/4}(\td^{1/2}_{2i}\wedge(\sigma\sigmabar)^{k-i})}{(-\lambda_\sigma)^{k-i}(k-i)!(n-k-i+1)!}\\
={}&\sum_{i=0}^k\frac{\Lambda_{\sigma/4}L_{\sigma}^{k-i}(\td^{1/2}_{2i})\wedge\sigmabar^{k-i}}{(-\lambda_\sigma)^{k-i}(k-i)!(n-k-i+1)!}\\
={}&\sum_{i=0}^k\frac{L_{\sigma}^{k-i}\Lambda_{\sigma/4}(\td^{1/2}_{2i})\wedge\sigmabar^{k-i}}{(-\lambda_\sigma)^{k-i}(k-i)!(n-k-i+1)!}\\
{}&-\sum_{i=0}^k\frac{(k-i)(k+i-n-1)L_{\sigma}^{k-i-1}(\td^{1/2}_{2i})\wedge\sigmabar^{k-i}}{(-\lambda_\sigma)^{k-i}(k-i)!(n-k-i+1)!}\\
={}&\sum_{i=1}^k\frac{L_{\sigma}^{k-i}(\td^{1/2}_{2i-2})\wedge\sigmabar^{k-i+1}}{\lambda_\sigma(-\lambda_\sigma)^{k-i}(k-i)!(n-k-i+1)!}\\
{}&-\sum_{i=0}^{k-1}\frac{-L_{\sigma}^{k-i-1}(\td^{1/2}_{2i})\wedge\sigmabar^{k-i}}{(-\lambda_\sigma)^{k-i}(k-i-1)!(n-k-i)!}\\
={}&0.
\end{align*}
Hence $\tp_{2k}$ is $\sigma$-primitive, in other words, $$\tp_{2k}\wedge\sigma^{n-2k+1}=L_{\sigma}^{n-2k+1}(\tp_{2k})=0.$$ As $\tp_{2k}$ is a real class, $\tp_{2k}\wedge\sigmabar^{n-2k+1}=0$ by complex conjugation. 
Hence it follows that $\tp_{2k}\wedge(\sigma+\sigmabar)^{2n-4k+1}=0$, that is, $\tp_{2k}$ is $(\sigma+\sigmabar)$-primitive. Here recall that $\sigma+\sigmabar$ is a K\"ahler form.
\end{proof}

\begin{remark}\label{primitive omega}
\begin{enumerate}
\item
Similar to Proposition~\ref{prop primitive}, by applying Corollary~\ref{lambda td2}, we can show that 
for any K\"{a}hler form $\omega$ and for $0\leq k\leq n/2$, 
$$
\sum_{i=0}^k\frac{(2n-4k+2)!(n-k-i+1)!\td^{1/2}_{2i}\wedge\omega^{2k-2i}}{(-\lambda(\omega))^{k-i}(k-i)!(n-2k+1)!(2n-2k-2i+2)!}\in H^{4k}(X)$$
is $\omega$-primitive.

\item Keep the notation in Section~\ref{sec IJK} and fix $0\leq k\leq n/2$. Then Proposition~\ref{prop primitive} shows that $\Lambda_{\sigma/4}(\tp_{2k})=0$ and $\Lambda_{\omega_J}(\tp_{2k})=0$. This implies that $\Lambda_{\omega_K}(\tp_{2k})=0$. In particular, $\tp_{2k}$ is both $\omega_J$-primitive and $\omega_K$-primitive. One may wonder whether $\tp_{2k}$ is $\omega_I$-primitive, but unfortunately this is not true if $n>1$ even for $k=1$.
Recall that by definition,
$$
\tp_{2}=\td^{1/2}_{2}-\frac{\sigma\sigmabar}{n\lambda_{\sigma}}.
$$
For computation we adopt notation in \cite{Veb-so5} or \cite[Section~24.2]{gross}, where the Lefschetz operators of $\omega_I, \omega_J, \omega_K$ are denoted by $L_I, L_J, L_K, \Lambda_I, \Lambda_J, \Lambda_K$.
By Corollary~\ref{lambda td2}, $\Lambda_{I}(\td^{1/2}_{2})=\frac{1}{\lambda(\omega_I)}\omega_I.$ On the other hand, 
\begin{align*}
\Lambda_I(\sigma\sigmabar)={}&\Lambda_I(\omega_J^2+\omega_K^2)\\
={}&L_J\Lambda_I(\omega_J)-K_{JI}(\omega_J)+L_K\Lambda_I(\omega_K)-K_{KI}(\omega_K)\\
={}&4\omega_I.
\end{align*}
Here we used the facts that $\Lambda_I(\omega_J)=\Lambda_I(\omega_K)=0$ and 
$K_{JI}(\omega_J)=K_{KI}(\omega_K)=-2\omega_I$ (see \cite[Proposition~24.2]{gross}). 
Note that $\lambda(\omega_I)=\lambda(\omega_J)=\frac{1}{4}\lambda_\sigma$. So we conclude that $
\Lambda_I(\tp_{2})=\frac{n-1}{n\lambda(\omega_I)}\omega_I\neq 0.
$
\end{enumerate}
\end{remark}

From the primitivity of $ \tp_{2k}$, we get the following.
\begin{cor}\label{cor pp=0}
Let $X$ be a hyperk\"{a}hler manifold of dimension $2n$ and fix a non-zero $\sigma\in H^0(X, \Omega^2_X)$.
Consider two integers $0\leq k, k'\leq n/2$ such that $k\neq k'$. Then
\begin{enumerate}
\item $\int \tp_{2k}\tp_{2k'}(\sigma\sigmabar)^{n-k-k'}=0;$ in particular, $\int \tp_{2k'}(\sigma\sigmabar)^{n-k'}=0$ if $k'\neq 0$.
\item 
$\int (\tp_{2k})^2(\sigma\sigmabar)^{n-2k}\geq 0;$ moreover, the equality holds if and only if $\tp_{2k}=0$.
\end{enumerate}
\end{cor}

\begin{proof}
(1) We may assume that $k>k'$, then $n-k-k'\geq n-2k+1$. By Proposition~\ref{prop primitive}, $\tp_{2k}\wedge\sigma^{n-k-k'}=0$. Hence the conclusion is clear.

(2) By Proposition~\ref{prop primitive}, $\tp_{2k}$ is $(\sigma+\sigmabar)$-primitive.
Note that $\sigma\sigmabar$ is of type $(4,0)+(2,2)+(0,4)$ with respect to the K\"{a}hler form $(\sigma+\sigmabar)$, so
 components of $\tp_{2k}$ in the Hodge decomposition with respect to the K\"{a}hler form $(\sigma+\sigmabar)$ are of types $(2k-2m, 2k+2m)$ for $-k\leq m\leq k$. On the other hand, $\tp_{2k}$ is real, 
hence by the Hodge--Riemann bilinear relation (\cite[Proposition 3.3.15]{huy-book}),
$$\int (\tp_{2k})^2(\sigma+\sigmabar)^{2n-4k}\geq 0,$$
where the equality holds if and only if $\tp_{2k}=0$.
This is equivalent to the conclusion by degree reason.
\end{proof}
The following is the main theorem of this section, which gives a Lefschetz-type decomposition of $\td^{1/2}$ in terms of $\tp_{2i}$. One special and important phenomenon is that the coefficients in this decomposition are all positive.
\begin{thm}\label{td=sum tp}Let $X$ be a hyperk\"{a}hler manifold of dimension $2n$ and fix a non-zero $\sigma\in H^0(X, \Omega^2_X)$. Consider $\lambda_\sigma=\lambda(\sigma+\sigmabar)>0$ as in Definition~\ref{def lambda}.
\begin{enumerate}
\item
For $k\leq n/2$,
\begin{align*}
\td^{1/2}_{2k}{}&=\sum_{i=0}^k\frac{(n-2k+i)!}{\lambda_\sigma^i i!(n-2k+2i)!}\tp_{2k-2i}\wedge(\sigma\sigmabar)^i\\
{}&=\sum_{i=0}^k\frac{(n-k-i)!}{\lambda_\sigma^{k-i} (k-i)!(n-2i)!}\tp_{2i}\wedge(\sigma\sigmabar)^{k-i}.
\end{align*}
\item For $k>n/2$,
$$
\td^{1/2}_{2k}=\sum_{i=0}^{n-k} \frac{(n-k-i)!}{\lambda_\sigma^{k-i} (k-i)!(n-2i)!}\tp_{2i}\wedge(\sigma\sigmabar)^{k-i}.
$$
\end{enumerate}
In summary, for any $0\leq k\leq n$, 
$$
\td^{1/2}_{2k}=\sum_{i=0}^{\min\{k, n-k\}} \frac{(n-k-i)!}{\lambda_\sigma^{k-i} (k-i)!(n-2i)!}\tp_{2i}\wedge(\sigma\sigmabar)^{k-i}.
$$
\end{thm}
\begin{proof}(1) This can be checked directly as the following.
\begin{align*}
{}&\sum_{i=0}^k\frac{(n-2k+i)!}{\lambda_\sigma^i i!(n-2k+2i)!}\tp_{2k-2i}\wedge(\sigma\sigmabar)^i\\
={}&\sum_{i=0}^k\frac{(n-2k+i)!}{\lambda_\sigma^i i!(n-2k+2i)!}
\left(\sum_{j=0}^{k-i}\frac{(n-2k+2i+1)!\td^{1/2}_{2j}\wedge(\sigma\sigmabar)^{k-i-j}\wedge(\sigma\sigmabar)^i}{(-\lambda_\sigma)^{k-i-j}(k-i-j)!(n-k+i-j+1)!}\right)\\
={}&\sum_{i=0}^k\sum_{j=0}^{k-i}
\frac{(-1)^{k-i-j}(n-2k+2i+1)(n-2k+i)!\td^{1/2}_{2j}\wedge(\sigma\sigmabar)^{k-j}}{\lambda_\sigma^{k-j} i!(k-i-j)!(n-k+i-j+1)!}\\
={}&\sum_{j=0}^k\frac{(-1)^{k-j}\td^{1/2}_{2j}\wedge(\sigma\sigmabar)^{k-j}}{\lambda_\sigma^{k-j}}\left(\sum_{i=0}^{k-j}
\frac{(-1)^{i}(n-2k+2i+1)(n-2k+i)!}{ i!(k-i-j)!(n-k+i-j+1)!}\right)\\
={}&\td^{1/2}_{2k}.
\end{align*}
Here in the last step, we applied Lemma~\ref{comb identity1} in the appendix.

(2) Note that by Corollary~\ref{lambda td} and Lemma~\ref{sigma sigmabar commutes},
$$
\Lambda_{\sigma/4}^{2k-n}(\td^{1/2}_{2k})=\frac{1}{\lambda_\sigma^{2k-n}}\td^{1/2}_{2n-2k}\wedge \sigmabar^{2k-n}
$$
and
\begin{align*}
{}&\Lambda_{\sigma/4}^{2k-n}\left(\sum_{i=0}^{n-k} \frac{(n-k-i)!}{\lambda_\sigma^{k-i} (k-i)!(n-2i)!}\tp_{2i}\wedge(\sigma\sigmabar)^{k-i}\right)\\
={}&\sum_{i=0}^{n-k} \frac{(n-k-i)!}{\lambda_\sigma^{k-i} (k-i)!(n-2i)!}\Lambda_{\sigma/4}^{2k-n}L_\sigma^{k-i}(\tp_{2i})\wedge\sigmabar^{k-i}\\
={}&\sum_{i=0}^{n-k} \frac{(n-k-i)!}{\lambda_\sigma^{k-i} (k-i)!(n-2i)!}\frac{(k-i)!^2}{(n-k-i)!^2}L_\sigma^{n-k-i}(\tp_{2i})\wedge\sigmabar^{k-i}\\
={}&\sum_{i=0}^{n-k} \frac{(k-i)!}{\lambda_\sigma^{k-i} (n-k-i)!(n-2i)!}\tp_{2i}\wedge \sigma^{n-k-i}\sigmabar^{k-i}\\
={}&\frac{1}{\lambda_\sigma^{2k-n}}\td^{1/2}_{2n-2k}\wedge \sigmabar^{2k-n}.
\end{align*}
Here for the second equality, we applied Lemma~\ref{AL=L} repeatedly $(2k-n)$ times; for the last one, we applied (1).
So the conclusion follows immediately as $\Lambda_{\sigma/4}^{2k-n}: H^{2k}(X,\Omega_X^{2k})\to H^{2k}(X, \Omega_X^{2n-2k})$ is an isomorphism by standard representation theory of $\mathfrak{sl}_2$.
\end{proof}

As a direct application of this decomposition, we recover an important result of Nieper-Wi{\ss}kirchen \cite{nieper-jag} generalizing Hitchin and Sawon \cite{hitchinsawon}. It was used by Huybrechts \cite{huybrechts1} to prove finiteness results for hyperk\"{a}hler manifolds.

\begin{cor}[{\cite[(5.17)]{nieper-jag}}]\label{cor cover nieper}
Let $X$ be a hyperk\"{a}hler manifold. Then
for any $\alpha\in H^2(X)$, 
$$\int\td^{1/2}(X)\exp(\alpha)=(1+\lambda(\alpha))^n\int\td^{1/2}(X).$$
\end{cor}
\begin{proof}
By Theorem~\ref{fujiki result}, it suffices to prove the result for $\alpha=\sigma+\sigmabar$.
By Theorem~\ref{td=sum tp}(1) and Corollary~\ref{cor pp=0}, for any $0\leq k\leq n$,
$$
\int \td^{1/2}_{2k}(\sigma\sigmabar)^{n-k}=\int \frac{(n-k)!}{\lambda_\sigma^{k} k!n!}(\sigma\sigmabar)^{n}
$$
as the integrals on components other than $\tp_0$ vanish. In particular,
\begin{align}\label{td1/2=int}
\int \td^{1/2}(X)=\int \td^{1/2}_{2n}=\frac{1}{\lambda_\sigma^n(n!)^2}\int (\sigma\sigmabar)^{n}.
\end{align}
Hence
\begin{align*}
\int \td^{1/2}_{2k}(\sigma\sigmabar)^{n-k}={}&\frac{(n-k)!}{\lambda_\sigma^{k} k!n!}\cdot \lambda_\sigma^n(n!)^2\int \td^{1/2}(X)\\
={}&\frac{(n-k)!n!}{k!}\cdot \lambda_\sigma^{n-k}\int \td^{1/2}(X).
\end{align*}
In other words,
$$
\int \td^{1/2}_{2k}\exp(\sigma+\sigmabar)=\binom{n}{k}\lambda_\sigma^{n-k}\int \td^{1/2}(X).$$
This concludes the proof.
\end{proof}
\section{Positivity of Riemann--Roch polynomials and applications}\label{final sec}
In this section, we study the positivity of the Riemann--Roch polynomials and its applications.
\subsection{Positivity of Riemann--Roch polynomials}
The following theorem is a more precise version of Theorem~\ref{main thm1 RR>0}.
\begin{thm}\label{thm td>0}
Let $X$ be a hyperk\"{a}hler manifold of dimension $2n$ and fix a non-zero $\sigma\in H^0(X, \Omega^2_X)$. Consider $\lambda_\sigma=\lambda(\sigma+\sigmabar)>0$ as in Definition~\ref{def lambda}.
Then for any $0\leq m\leq n$,
$$
\int \td_{2m} \exp(\sigma+\sigmabar)\geq \binom{2n-m+1}{m}\lambda_\sigma^{n-m}\int\td^{1/2}(X).
$$
Moreover, the inequality is strict for $m>1$ and $n>1$.
\end{thm}
\begin{proof}
Note that by definition, $\td_{2m}=\sum_{k=0}^m\td^{1/2}_{2k}\td^{1/2}_{2m-2k}$. Hence by Theorem~\ref{td=sum tp} and Corollary~\ref{cor pp=0},
\begin{align*}
{}&\int\td_{2m}(\sigma\sigmabar)^{n-m}\\
={}&\int\sum_{k=0}^m\left(\sum_{i=0}^{\min\{k, n-k\}} \frac{(n-k-i)!}{\lambda_\sigma^{k-i} (k-i)!(n-2i)!}\tp_{2i}(\sigma\sigmabar)^{k-i}\right)\\
{}&\cdot \left(\sum_{i=0}^{\min\{m-k, n-m+k\}} \frac{(n-m+k-i)!}{\lambda_\sigma^{m-k-i} (m-k-i)!(n-2i)!}\tp_{2i}(\sigma\sigmabar)^{m-k-i}\right)(\sigma\sigmabar)^{n-m}\\
={}&\sum_{k=0}^m\sum_{i=0}^{\min\{k, m-k\}} \frac{(n-k-i)!(n-m+k-i)!}{\lambda_\sigma^{m-2i} (k-i)!(m-k-i)!(n-2i)!^2}\int(\tp_{2i})^2(\sigma\sigmabar)^{n-2i}\\
={}&\sum_{i=0}^{\rounddown{m/2}}\sum_{k=i}^{m-i} \frac{(n-k-i)!(n-m+k-i)!}{\lambda_\sigma^{m-2i} (k-i)!(m-k-i)!(n-2i)!^2}\int(\tp_{2i})^2(\sigma\sigmabar)^{n-2i}\\
={}&\sum_{i=0}^{\rounddown{m/2}} \frac{(n-m)!^2}{\lambda_\sigma^{m-2i} (n-2i)!^2}\binom{2n-2i-m+1}{m-2i}\int(\tp_{2i})^2(\sigma\sigmabar)^{n-2i}\\
\geq{}& \frac{(n-m)!^2}{\lambda_\sigma^m n!^2}\binom{2n-m+1}{m}\int (\sigma\sigmabar)^{n}\\
={}&(n-m)!^2\binom{2n-m+1}{m}\lambda_\sigma^{n-m}\int\td^{1/2}(X).
\end{align*}
Here in the last three steps we applied Lemma~\ref{comb identity2}, Corollary~\ref{cor pp=0}(2), and Equality~\eqref{td1/2=int}.
This proves the desired inequality.

If the equality holds for some $m>1$, then by Corollary~\ref{cor pp=0}(2), $\tp_2=0$.
Then Theorem~\ref{td=sum tp} implies that $\td^{1/2}_2$ is proportional to $\sigma\sigmabar$, which is absurd if $n>1$, as $\td^{1/2}_2$ does not depend on the complex structure of $X$.

Finally we remark that, from the above expression, if one could get a better estimate for $\int(\tp_{2i})^2(\sigma\sigmabar)^{n-2i}$
for $i>0$, then we can get a better estimate for $\int\td_{2m}(\sigma\sigmabar)^{n-m}$.
\end{proof}
\begin{cor}\label{cor P>0}
Let $X$ be a hyperk\"{a}hler manifold of dimension $2n>2$. 
Then for any $\alpha\in H^2(X)$, 
$$
P_{\lambda(\alpha)}:=\int \td(X) \exp(\alpha)-\sum_{m=0}^n\binom{2n-m+1}{m}\lambda(\alpha)^{n-m}\int\td^{1/2}(X)
$$
is a polynomial in terms of $\lambda(\alpha)$ of degree $n-2$ with positive coefficients.
\end{cor}
\begin{proof}
By Theorem~\ref{fujiki result},
the coefficient of $\lambda(\alpha)^{n-m}$ in $P_{{\lambda(\alpha)}}$ is just 
$$
b_{n-m}=\frac{1}{\lambda_{\sigma}^{n-m}}\int \td_{2m} \exp(\sigma+\sigmabar) - \binom{2n-m+1}{m}\int\td^{1/2}(X).
$$

If $m>1$, then $b_{n-m}>0$ by Theorem~\ref{thm td>0}.
If $m=0$, then $b_{n}=0$ by Equality~\eqref{td1/2=int}.
If $m=1$, then by the definition of $\lambda_\sigma,$
\begin{align*}
b_{n-1}{}&=\frac{1}{\lambda_{\sigma}^{n-1}}\int \td_{2} \exp(\sigma+\sigmabar) - 2n\int\td^{1/2}(X)\\
{}&=\frac{1}{\lambda_{\sigma}^{n-1}}\int \frac{1}{12}c_2(X) \exp(\sigma+\sigmabar) - 2n\frac{1}{\lambda_{\sigma}^n}\int\exp(\sigma+\sigmabar)=0.
\end{align*}
Hence $P_{\lambda(\alpha)}$ is a polynomial in terms of $\lambda(\alpha)$ of degree $n-2$ with positive coefficients.
\end{proof}

\begin{proof}[Proof of Theorem~\ref{main thm1 RR>0}]
This follows from Proposition~\ref{prop lambda=q} and Corollary~\ref{cor P>0}.
\end{proof}

\subsection{Kawamata's effective non-vanishing conjecture and Riess's question}
Recall that a special version of Kawamata's effective non-vanishing conjecture predicts that, if $L$ is a nef and big line bundle on a projective manifold $X$ with $c_1(X)=0$, then $h^{0}(X, L)>0$. In \cite{CJmathz} we studied this conjecture and proposed a stronger version for projective hyperk\"{a}hler manifolds (\cite[Conjecture 3.6]{CJmathz}), which is actually equivalent to Theorem~\ref{main thm1 RR>0} for projective hyperk\"{a}hler manifolds. So by Theorem~\ref{main thm1 RR>0}, we get the following corollary.

\begin{cor}[{\cite[Conjecture 3.6]{CJmathz}}]\label{cj conj}
Let $X$ be a projective hyperk\"{a}hler manifold of dimension $2n$ and $L$ a nef and big line bundle on $X$. Then
\begin{enumerate}
\item $h^{0}(X,L)\geq n+2$;
\item $\int\td_{2n-2i}(X)\cdot L^{2i}> 0$ for all $0\leq i\leq n$.
\end{enumerate}
\end{cor}
\begin{proof}

(2) directly follows from Theorem~\ref{thm td>0} and Theorem~\ref{fujiki result}.
For (1), by the
Kawamata--Viehweg
vanishing theorem (\cite{kawamata1}), Theorem~\ref{main thm1 RR>0}, and Proposition~\ref{lemma qL>0},
$$
h^0(X, L)=\chi(L)>\chi(\OO_X)=n+1.
$$
Here recall that the constant term of $\RR_X$ is just $\chi(\OO_X)$.
\end{proof}

As a related topic, Riess \cite{Riess} studied the base loci of linear systems of line bundles on hyperk\"{a}hler manifolds and naturally raised up the question whether the Riemann--Roch polynomial $RR_X(q)|_{q>0}$ is strictly monotonic.
Theorem~\ref{main thm1 RR>0} answers her question affirmatively.

\begin{cor}[Riess's question]\label{riess conj}
Let $X$ be a hyperk\"{a}hler manifold. Then the Riemann--Roch polynomial $RR_X(q)$ is strictly monotonic for $q>0$. 
\end{cor}
\subsection{An upper bound of $\int \td^{1/2}(X)$}
As an application of Theorem~\ref{thm td>0}, we can give an upper bound for the value $\int \td^{1/2}(X)$. 
\begin{cor}\label{upper td1/2}
Let $X$ be a hyperk\"{a}hler manifold of dimension $2n>2$. 
Then
$
\int\td^{1/2}(X)< 1.
$
Equivalently, let $g$ be a hyperk\"{a}hler metric on $X$ compatible with the hyperk\"{a}hler structure on $X$,
then $$
||R||^{2n}<(192\pi^2n)^n(\vol\, X)^{n-1},
$$
where $||R||$ is the $\text{\rm L}_2$-norm of the curvature tensor of $g$.
\end{cor}
\begin{proof}
In Theorem~\ref{thm td>0}, taking $m=n$, we get
$$
\int\td^{1/2}(X)< \frac{1}{n+1} \int \td_{2n} = \frac{1}{n+1} \chi(\OO_X)=1.
$$
The second statement follows directly from \cite[Theorem 5]{hitchinsawon}.
\end{proof}

\begin{example}
\begin{enumerate}
\item For a K3 surface $S$, $\int\td^{1/2}(S)= c_2(S)/24=1.$ 
\item If $X$ is the Hilbert scheme of $n$ points on a K3 surface, then $\int\td^{1/2}(X)=\frac{(n+3)^n}{4^nn!}$ by Sawon \cite[Proposition 19]{Sawon-phd}.
\item If $X$ is a generalized Kummer variety of dimension $2n$, then $\int\td^{1/2}(X)=\frac{(n+1)^{n+1}}{4^nn!}$ by Sawon \cite[Proposition 21]{Sawon-phd}.
\item As all Chern numbers of O'Grady's examples are known due to \cite{mrs} ($6$-dimensional case) and \cite{ortiz} ($10$-dimensional case), we can compute $\int\td^{1/2}(X)$ for these examples. In fact, Belmans informed the author that he computed that $\int\td^{1/2}(X)=\frac{2}{3}$ or $\frac{4}{15}$ for O'Grady's $6$-dimensional example and $10$-dimensional example respectively, which coincides with the value for a generalized Kummer variety of dimension $6$, or the Hilbert scheme of $5$ points on a K3 surface respectively. Then we realized that all the above numbers can be obtained directly by Example~\ref{ex RR} and the following Lemma~\ref{lem td same}.
\end{enumerate}
\end{example}
\begin{lem}\label{lem td same}
Let $X$ be a hyperk\"{a}hler manifold of dimension $2n$ and suppose that the first two leading terms of $\RR_X(q)$ are $Aq^n$ and $Bq^{n-1}$, then $\int\td^{1/2}(X)=\frac{B^n}{(2n)^nA^{n-1}}$. In particular, if $X$ and $Y$ are two hyperk\"{a}hler manifolds with the same Riemann--Roch polynomial $\RR_X(q)=\RR_Y(q)$, then
$\int\td^{1/2}(X)=\int\td^{1/2}(Y)$.
\end{lem}

\begin{proof}
Write $\RR_X(q)=Aq^n+Bq^{n-1}+(\text{lower terms})$. 
Recall that by \cite[Proof of Lemma~3]{ortiz}, $c_X=(2n)!A$ and $\lambda(\sigma+\sigmabar)=\frac{2nA}{B}q_X(\sigma+\sigmabar)$, where $c_X$ is the Fujiki constant and $\lambda$ is in Definition~\ref{def lambda}. Then by \eqref{td1/2=int},
\begin{align*}
\int \td^{1/2}(X)={}&\frac{1}{\lambda(\sigma+\sigmabar)^n(n!)^2}\int (\sigma\sigmabar)^{n}=\frac{1}{\lambda(\sigma+\sigmabar)^n(2n)!}\int (\sigma+\sigmabar)^{2n}\\
={}&\frac{c_Xq_X(\sigma+\sigmabar)^n}{\lambda(\sigma+\sigmabar)^n(2n)!}=\frac{B^n}{(2n)^nA^{n-1}}.
\end{align*}
\end{proof}
The examples suggest that $\int\td^{1/2}(X)$ might get very small as $n$ getting large, so it is natural to ask whether there is a better upper bound for $\int\td^{1/2}(X)$ of exponential order $c<0$ in terms of $\dim X$.

Recall that for a hyperk\"{a}hler manifold $X$ of dimension $2n$, its {\it Chern numbers} are given by integrals of the form $\int c_{2k_1}c_{2k_2}\dots c_{2k_m}$ for non-negative integers $k_1,\dots, k_m$ satisfying $\sum_{i=1}^mk_i=n$.
As observed by Sawon \cite{Sawon-phd} and Nieper-Wi{\ss}kirchen \cite[Appendix B]{nieper-phd} (see also \cite{mrs} and \cite{ortiz}), all known Chern numbers of hyperk\"{a}hler manifolds are positive. So it is natural to propose the following conjecture, which is a question by Nieper-Wi{\ss}kirchen \cite[Appendix B]{nieper-phd}.
\begin{conj}\label{conj chern>0}
Let $X$ be a hyperk\"{a}hler manifold of dimension $2n$. 
Then all Chern numbers $\int c_{2k_1}c_{2k_2}\dots c_{2k_m}$ for non-negative integers $k_1,\dots, k_m$ satisfying $\sum_{i=1}^mk_i=n$ are positive integers.
\end{conj}
If this conjecture is true, then it reflects very special geometry of hyperk\"{a}hler manifolds. For example, it predicts that the topological Euler characteristic of any hyperk\"{a}hler manifold is positive as a special case, which is unfortunately unknown even in dimension $4$. One can expect that the methods in this paper might give some partial solutions to this conjecture.

\appendix

\section{Some combinatorial identities}
In this appendix, we prove two combinatorial identities.
\begin{lem}\label{comb identity1}Given non-negative integers $n,k,j$ satisfying $n/2\geq k\geq j$, we have
$$\sum_{i=0}^{k-j}
\frac{(-1)^{i}(n-2k+2i+1)(n-2k+i)!}{ i!(k-i-j)!(n-k+i-j+1)!}=
\begin{cases}
0 &\text{if } k>j;\\
1 &\text{if } k=j.
\end{cases}$$
\end{lem}
\begin{proof}
The case when $k=j$ is trivial. Suppose that $k>j$.
The desired equality is equivalent to 
$$
\sum_{i=0}^{k-j}(-1)^{i}(n-2k+2i+1)\binom{n-2k+i}{i}\binom{n-2j+1}{k-i-j}=0.
$$
Note that 
\begin{align*}
{}&\sum_{i=0}^{k-j}(-1)^{i}(n-2k+2i+1)\binom{n-2k+i}{i}\binom{n-2j+1}{k-i-j}\\
={}&\sum_{i=0}^{k-j}(-1)^{i}((n-2k+i+1)+i)\binom{n-2k+i}{i}\binom{n-2j+1}{k-i-j}\\
={}&\sum_{i=0}^{k-j}(-1)^{i}(n-2k+1)\binom{n-2k+i+1}{i}\binom{n-2j+1}{k-i-j}\\
{}&+\sum_{i=1}^{k-j}(-1)^{i}(n-2k+1)\binom{n-2k+i}{i-1}\binom{n-2j+1}{k-i-j}.
\end{align*}
To conclude the proof, we claim that
\begin{align*}
{}&\sum_{i=0}^{k-j}(-1)^{i}\binom{n-2k+i+1}{i}\binom{n-2j+1}{k-i-j}\\
={}&\sum_{i=1}^{k-j}(-1)^{i-1}\binom{n-2k+i}{i-1}\binom{n-2j+1}{k-i-j}\\
={}&\binom{2k-2j-1}{k-j}.
\end{align*}
In fact, the first item is the coefficient of $x^{k-j}$ in the generating function
$$
(1+x)^{-(n-2k+2)}\cdot (1+x)^{n-2j+1}=(1+x)^{2k-2j-1},
$$
so it equals to $\binom{2k-2j-1}{k-j}$; meanwhile, the second item is the coefficient of $x^{k-j-1}$ in the generating function
$$
(1+x)^{-(n-2k+2)}\cdot (1+x)^{n-2j+1}=(1+x)^{2k-2j-1},
$$
so it equals to $\binom{2k-2j-1}{k-j-1}=\binom{2k-2j-1}{k-j}$.
\end{proof}

\begin{lem}\label{comb identity2}Given non-negative integers $n,m$ satisfying $n\geq m$, we have
$$\sum_{k=0}^m \frac{(n-k)!(n-m+k)!}{k!(m-k)!}=(n-m)!^2\binom{2n-m+1}{m}.$$
Furthermore, if $i$ is an integer satisfying $m\geq 2i$, then
$$\sum_{k=i}^{m-i} \frac{(n-k-i)!(n-m+k-i)!}{(k-i)!(m-k-i)!}=(n-m)!^2\binom{2n-2i-m+1}{m-2i}.$$
\end{lem}
\begin{proof}
Consider 
$$\frac{1}{(n-m)!^2}\sum_{k=0}^m \frac{(n-k)!(n-m+k)!}{k!(m-k)!}=\sum_{k=0}^m \binom{n-k}{m-k}\binom{n-m+k}{k}.$$
This is exactly the coefficient of $x^{m}$ in the generating function
$$
(1-x)^{-(n-m+1)}\cdot (1-x)^{-(n-m+1)}=(1-x)^{-(2n-2m+2)}, 
$$
which is just $\binom{2n-m+1}{m}.$
The second equality follows from the first one by considering $n-2i$ and $m-2i$, and changing the range of $k$ to $[0, m-2i].$
\end{proof}

\section*{Acknowledgments}
The author is grateful to Yalong Cao for inspiration and fruitful discussions.
The author would like to thank Thorsten Beckmann, Pieter Belmans, \'{A}ngel Ortiz, and Shilin Yu for discussions and comments during the preparation of this paper. 
The author would like to thank the referee for useful comments and suggestions.
The author was supported by National Key Research and Development Program of China (Grant No.~2020YFA0713200).


\begin{thebibliography}{99}

 
\bibitem{beauville} A.~Beauville, {\it Vari\'{e}t\'{e}s K\"{a}hleriennes dont la premi\`{e}re classe de Chern est nulle}, J. Differential Geom. {\bf 18} (1983), no. 4, 755--782.
 
\bibitem{bogomolov} F.A.~Bogomolov, \textit{Hamiltonian K\"{a}hlerian manifolds}, Dokl. Akad. Nauk SSSR {\bf 243} (1978), no. 5, 1101--1104.

\bibitem{BN} M.~Britze, M.A.~Nieper, \textit{Hirzebruch--Riemann--Roch formulae on irreducible symplectic K\"{a}hler manifolds}, arXiv:math/0101062v1.

\bibitem{CJmathz} Y.~Cao, C.~Jiang, {\it Remarks on Kawamata’s effective non-vanishing conjecture for manifolds with trivial first Chern classes}, Math. Z. {\bf 296} (2020), 615--637.




\bibitem{egl} G.~Ellingsrud, L.~G\"{o}ttsche, M.~Lehn, \textit{On the cobordism class of the Hilbert scheme of a surface}, J. Algebraic Geom. {\bf 10} (2001), no. 1, 81--100.

\bibitem{fujiki} A.~Fujiki, \textit{On the de Rham cohomology group of a compact K\"{a}hler symplectic manifold}, Algebraic geometry, Sendai, 1985, 105--165, Adv. Stud. Pure Math., 10, North-Holland, Amsterdam, 1987. 



\bibitem{gross} M.~Gross, D.~Huybrechts, D.~Joyce, \textit{Calabi--Yau manifolds and related geometries}, Lectures from the Summer School held in Nordfjordeid, June 2001. Universitext. Springer-Verlag, Berlin, 2003.

\bibitem{guan} D. Guan, \textit{On the Betti numbers of irreducible compact hyperk\"{a}hler manifolds of complex dimension four}, Math. Res. Lett. {\bf 8} (2001), no. 5-6, 663--669.



\bibitem{hitchinsawon} N.~Hitchin, J.~Sawon, \textit{Curvature and characteristic numbers of hyper-K\"{a}hler manifolds}, Duke Math. J. {\bf 106} (2001), no. 3, 599--615.


\bibitem{huy97} D.~Huybrechts, {\it Compact hyperk\"{a}hler manifolds}, Habilitationsschrift Essen (1997), 65 pages.

\bibitem{huybrechts} D.~Huybrechts, \textit{Compact hyper-K\"{a}hler manifolds: basic results}, Invent. Math. {\bf 135} (1999), no. 1, 63--113; {\it Erratum: ``Compact hyper-K\"{a}hler manifolds: basic results"}, Invent. Math. {\bf 152} (2003), no. 2, 209--212.

\bibitem{huybrechts1} D. Huybrechts, \textit{Finiteness results for compact hyperk\"{a}hler manifolds}, J. Reine Angew. Math. {\bf 558} (2003), 15--22.

\bibitem{huy-book} D.~Huybrechts, \textit{Complex geometry. An introduction}. Universitext. Springer-Verlag, Berlin, 2005. xii+309 pp. 
 
 
 
 \bibitem{Kapranov} M.~Kapranov, {\it Rozansky--Witten invariants via Atiyah classes}, Compositio Math. {\bf 115} (1999), no. 1, 71--113.
 
\bibitem{kawamata1} Y.~Kawamata, \textit{A generalization of Kodaira--Ramanujam's vanishing theorem}. Math. Ann. {\bf 261} (1982), no. 1, 43--46.


 \bibitem{LL} E.~Looijenga, V.A.~Lunts, {\it A Lie algebra attached to a projective variety},
Invent. Math. {\bf 129} (1997), no. 2, 361--412.






\bibitem{mrs} G. Mongardi, A. Rapagnetta, G. Sacc\`{a}, \textit{The Hodge diamond of O'Grady's $6$-dimensional example}, Compos. Math. {\bf 154} (2018), no. 5, 984--1013.


\bibitem{nieper-phd}M.A.~Nieper-Wi{\ss}kirchen, {\it Characteristic classes and Rozansky--Witten invariants of
compact hyperk\"{a}hler manifolds}, Ph.D Thesis, K\"oln 2002.

\bibitem{nieper-jag}M.A.~Nieper, \textit{Hirzebruch--Riemann--Roch formulae on irreducible symplectic K\"{a}hler manifolds}, J. Algebraic Geom. {\bf 12} (2003), no. 4, 715--739.

\bibitem{nieper-book}M.A.~Nieper-Wi{\ss}kirchen, \textit{Chern numbers and Rozansky--Witten invariants of compact hyper-K\"{a}hler manifolds}, World Scientific Publishing Co., Inc., River Edge, NJ, 2004.

\bibitem{ogrady1} K.G.~O'Grady, \textit{Desingularized moduli spaces of sheaves on a $K3$}, J. Reine Angew. Math. {\bf 512} (1999), 49--117.

\bibitem{ogrady2} K.G.~O'Grady, \textit{A new six-dimensional irreducible symplectic variety}, J. Algebraic Geom. {\bf 12} (2003), no. 3, 435--505.

\bibitem{ortiz} \'{A}.D.R.~Ortiz, {\it Riemann--Roch polynomials of the known hyperk\"{a}hler manifolds}, with an appendix by Yalong Cao and Chen Jiang, arXiv:2006.09307v2.


\bibitem{Riess}U.~Riess, {\it Base divisors of big and nef line bundles on irreducible symplectic varieties}, arXiv:1807.05192v1, appear in Ann. Inst. Fourier (Grenoble).


\bibitem{RW}L.~Rozansky, E.~Witten, {\it Hyper-K\"{a}hler geometry and invariants of three-manifolds},
Selecta Math. (N.S.) {\bf 3} (1997), no. 3, 401--458.


\bibitem{Sawon-phd} J.~Sawon, {\it Rozansky--Witten invariants of hyperk\"{a}hler manifolds}, Ph.D. thesis, University of Cambridge, arXiv:math/0404360v1, October 1999.

\bibitem{Thurston} D.P.~Thurston, {\it Wheeling: A diagrammatic analogue of the Duflo isomorphism}, Ph.D. thesis, University of California at Berkeley, arXiv:math/0006083v1, Spring 2000.

\bibitem{Veb-so5} M.~Verbitsky, {\it Action of the Lie algebra of $SO(5)$ on the cohomology of a hyper-K\"{a}hler manifold},
Funktsional. Anal. i Prilozhen. {\bf 24} (1990), no. 3, 70--71; translation in
Funct. Anal. Appl. 24 (1990), no. 3, 229--230 (1991).

\bibitem{Veb-coh} M.~Verbitsky, {\it Cohomology of compact hyper-K\"{a}hler manifolds and its applications}, Geom. Funct. Anal. {\bf 6} (1996), no. 4, 601--611.




\bibitem{yau} S. T. Yau, \textit{On the Ricci curvature of a compact K\"{a}hler manifold and the complex
Monge--Amp\`{e}re equation. I}, Comm. Pure Appl. Math. {\bf 31} (1978), no. 3, 339--411.

 \end{thebibliography}
\end{document}